\documentclass[a4paper, 
headsepline=off, DIV=12, titlepage=false]{scrartcl}

\title{\vspace{-1em}On universal norms for $\bm{p}$-adic representations in higher-rank Iwasawa theory}
\date{}
\author{Dominik Bullach \and Alexandre Daoud\vspace{-2em}}
 
 \usepackage[backend=bibtex
 , style=alphabetic, sorting=nyt, maxbibnames=9
 ]{biblatex}
\renewbibmacro{in:}{%
  \ifentrytype{article}{}{\printtext{\bibstring{in}\intitlepunct}}}
  \DeclareFieldFormat{journaltitle}{#1\isdot}
  \DeclareFieldFormat[article, inproceedings, unpublished]{title}{\textit{#1}}

	\usepackage{amsmath}	
	\usepackage{amssymb}
	\usepackage{amsthm}
	\usepackage[british]{babel}
	\usepackage{comment}
	\usepackage{datetime}
	\usepackage{enumitem}
	\usepackage[inner=30mm, outer=30mm, 
	bottom=28mm,
				a4paper]{geometry}
	\usepackage[framemethod=TikZ, roundcorner=1mm, linecolor=gray]{mdframed}
	\usepackage{HanFakt} 
	\usepackage{tikz-cd}
	\usepackage{url}
	\usepackage{stmaryrd}
	\usepackage{xcolor}
	\usepackage{mathrsfs}
	\usepackage{mathtools}
	\usepackage{todonotes}
	\usepackage{bm}

\usepackage{tikz-cd, tikz-3dplot, pgfplots}
\usetikzlibrary{patterns, arrows, positioning, decorations.markings}

	\tikzset{graph/.style={grau}}
\tikzcdset{
  cells={font=\everymath\expandafter{\the\everymath\displaystyle}},
}

\usepackage[automark, autooneside=false]{scrlayer-scrpage}

\newpairofpagestyles[scrheadings]{default}{
\automark[subsection]{section}
\clearpairofpagestyles
\chead{\leftmark : \rightmark}
\cfoot*{\pagemark}
}

\newpairofpagestyles[scrheadings]{special}{
\automark[subsection]{section}
\clearpairofpagestyles
\chead{\leftmark}
\cfoot*{\pagemark}
}

\newcommand{\nocontentsline}[3]{}
\newcommand{\tocless}[2]{\bgroup\let\addcontentsline=\nocontentsline#1{#2}\egroup}

	\renewmdenv[%
   	 backgroundcolor=gray!25,
   	 linecolor=white,
    	outerlinewidth=1pt,
    	roundcorner=3mm,
	]{boxed}	

	\swapnumbers
	\theoremstyle{definition}
	\newtheorem*{thmplain}{Theorem}
	\newtheorem*{propplain}{Proposition}
	\newtheorem{thm}{Theorem}[section]
	
	\newtheorem{prop}[thm]{Proposition}
	\newtheorem{cor}[thm]{Corollary}
	\newtheorem{lem}[thm]{Lemma}
	
	\newtheorem{bsp1}[thm]{Example}
	\newtheorem{bspe1}[thm]{Examples}
	\newtheorem{definition}[thm]{Definition}
	\newtheorem{rk}[thm]{Remark}

	\newtheorem{lemma}[thm]{Lemma}
	\newtheorem{corollary}[thm]{Corollary}
	\newtheorem{proposition}[thm]{Proposition}
	\newtheorem{conjecture}[thm]{Conjecture}

	\newtheorem{examples}[thm]{Examples}
	\newtheorem{remark}[thm]{Remark}
	
	\newtheorem*{acknowledgments}{Acknowledgements}
	\newtheorem{hypothesis}[thm]{Hypothesis}
	\newtheorem{hypotheses}[thm]{Hypotheses}

	 \newenvironment{proofbox}
	  {
	  \begin{proof}}
	 { \end{proof} 
	 \medskip}

	\newenvironment{liste}
	{\begin{enumerate}[label=(\alph*)]}
	{\end{enumerate}}

	\newenvironment{cdiagram}
	{\begin{center} \begin{tikzcd}}
	{\end{tikzcd} \end{center}}

	\renewcommand{\:}{\colon}

	\newcommand{\bigO}{\mathcal{O}}

	\newcommand{\cl}{\operatorname{cl}}
	\newcommand{\C}{\mathbb{C}}
	\renewcommand{\emptyset}{\varnothing}
	\renewcommand{\emph}[1]{\textit{\textbf{#1}}}
	\newcommand{\Fitt}{\operatorname{Fitt}}
	\newcommand{\Frob}{\mathrm{Frob}}
	\newcommand{\gal}[2]{\text{Gal}( #1 | #2)}
	\renewcommand{\iff}{\quad \Leftrightarrow \quad}

	 \newcommand{\N}{\mathbb{N}}
	
	\newcommand{\p}{\mathfrak{p}}

	\newcommand{\q}{\mathfrak{q}}
	\newcommand{\Q}{\mathbb{Q}}

	\newcommand{\R}{\mathbb{R}}

	\newcommand{\UN}{\mathrm{UN}}
	\newcommand{\Z}{\mathbb{Z}}

	\renewcommand{\thethm}{(\arabic{section}.\arabic{thm})}

\usepackage[colorlinks,
pdfpagelabels,
pdfstartview = FitH,
bookmarksopen = true,
bookmarksnumbered = true,
linkcolor = black,
plainpages = false,
hypertexnames = false,
citecolor = black, urlcolor=black]{hyperref}

  1

\DeclareMathOperator{\Det}{Det}
\DeclareMathOperator{\Ext}{Ext}
\DeclareMathOperator{\Tor}{Tor}
\DeclareMathOperator{\Gal}{Gal}

\DeclareMathOperator{\Hom}{Hom}

\DeclareMathOperator{\Spec}{Spec}

\DeclareMathOperator{\coker}{coker}
\DeclareMathOperator{\res}{res}
\DeclareMathOperator{\im}{im}

\newcommand{\NN}{\mathbb{N}}
\newcommand{\QQ}{\mathbb{Q}}

\newcommand{\ZZ}{\mathbb{Z}}

\newcommand{\cQ}{\mathcal{Q}}

\newcommand{\cG}{\mathcal{G}}
\newcommand{\cK}{\mathcal{K}}

\newcommand{\cO}{\mathcal{O}}

\newcommand{\fp}{\mathfrak{p}}

\newcommand{\cone}{\mathrm{cone}}
\newcommand{\et}{\text{\'et}}

\newcommand{\Ann}{\mathrm{Ann}}

\def\bigcapp{\raise1ex\hbox{\rotatebox{180}{$\biguplus$}}}
 \def\bigcappd{\raise1ex\hbox{\rotatebox{180}{$\displaystyle\biguplus$}}}

 \newcommand{\tor}{\mathrm{tor}}
 \newcommand{\tf}{\mathrm{tf}}

 \newcommand{\cR}{\mathcal{R}}
 \newcommand{\ram}{\mathrm{ram}}
 \newcommand{\cores}{\mathrm{cores}}
 \newcommand{\bidual}{\bigcap\nolimits}
\newcommand{\exprod}{\bigwedge\nolimits}
\newcommand{\rank}{\mathrm{rk}}
\newcommand{\spc}{\mathrm{split}}
\newcommand{\NS}{\mathrm{NS}}
\newcommand{\Iw}{\mathrm{Iw}}
\newcommand{\Cyc}{\mathrm{Cyc}}
\newcommand{\cchar}{\mathrm{char}}

  \tikzset{
    rotated/.style={rotate=-90, anchor = south},
    rotatedswap/.style={rotate=-90, anchor=north, outer sep=0.75mm}
}

\newcommand{\bLambda}{{\mathpalette\makebLambda\relax}}
\newcommand{\makebLambda}[2]{%
  \raisebox{\depth}{\scalebox{1}[-1]{$\mathsurround=0pt#1\mathbb{V}$}}%
}

\usepackage{dsfont}
\renewcommand{\mathbb}{\mathds}

\begin{document}

\maketitle

\begin{abstract}
    We begin a systematic investigation of universal norms for $p$-adic representations in higher-rank Iwasawa theory. After establishing the basic properties of the module of higher-rank universal norms, that naturally extend those in the classical theory, we construct an Iwasawa-theoretic pairing that is relevant to this setting. This allows us, for example, to refine the classical Iwasawa Main Conjecture for cyclotomic fields, and also to give applications to various well-known conjectures in arithmetic concerning Iwasawa invariants and leading terms of $L$-functions. 
\end{abstract}

\let\thefootnote\relax

\tableofcontents

\section{Introduction}
\pagestyle{special}

The investigation of the deep connection between $L$-functions and arithmetic is at the heart of modern number theory. By now we have a number of partial results on this matter, many of which due to celebrated results obtained via the Euler system method that has been developed by Thaine, Kolyvagin, Rubin and Mazur.\\
However, all of these (unconditional) results are restricted to cases where the order of vanishing of the $L$-function is at most one. Although a notion of \textit{higher-rank} Euler system was already established by Perrin-Riou more than 20 years ago, technical issues arising from the use of exterior powers 
hindered the theory surrounding higher rank Euler systems from being fully operational. These technical obstructions have only recently been overcome by Burns, Sakamoto and Sano in a series of articles (\cite{EulerSystemsSagaI},\ \cite{EulerSystemsSagaII},\ \cite{EulerSystemsSagaIII} and \cite{EulerSystemsSagaIV})\@. Key to their approach is the consistent use of \textit{exterior biduals} instead of exterior powers, a notion that is based on the lattice introduced by Rubin in \cite[\S 1.2]{Rubin96} and provides better functorial properties in many aspects.\\
Since Euler systems are, by their very definition, universal norms on $\ZZ_p$-extensions, we feel that the study of \textit{higher-rank universal norms} 
undertaken in this article naturally fits into the chain of developments described above. As in the aforementioned works, the use of exterior biduals allows us to develop a theory that naturally extends the classical theory of universal norms to both the higher-rank and equivariant settings. 

\paragraph{Overview of results} To explain our results in a little more detail, we first introduce some notation. Let $L | K$ be a finite abelian extension of number fields in which every archimedean place splits completely, $p$ an odd prime and 
take $L_\infty = \bigcup_{n \geq 0} L_n$ to be a $\Z_p$-extension of $L$ that is abelian over $K$ and in which no finite place of $K$ splits completely. Denote by 
$\Lambda = \Z_p \llbracket \gal{L_\infty}{L} \rrbracket$ and $\bLambda = \Z_p \llbracket \gal{L_\infty}{K} \rrbracket$ the relevant Iwasawa algebras. 
For a $p$-adic representation $T$ of $K$ we shall below define natural modules $\UN^r_n (T)$ and $\NS^r (T)$ of universal norms and norm-coherent sequences, respectively, of rank $r$ and level $n$ along the $\Z_p$-extension $L_\infty |L$. \\
Our first step in extending the classical theory of universal norms as established by, for example, Kuz'min \cite{Kuzmin} and Greither \cite{Greither} (see Remark \ref{structure-theorem-remark}\,(a) for more details on the existing literature) is then the following theorem.

\begin{thmplain}[Thm.\@ \ref{UN-structure-theorem}]
    Fix an integer $1 \leq r \leq r_T$ where $r_T$ denotes the \textit{basic rank} of $T$ (see \ref{main-hypothesis}\,(2)). Then, under certain mild conditions, the natural codescent map induces an isomorphism of $\Z_p [\gal{L_n}{K}]$-modules
    \begin{align*}
         \NS^r \otimes_{\bLambda} \Z_p [\gal{L_n}{K}] \cong \UN^r_n (T).
    \end{align*}
    Moreover, $\NS^r$ (resp.\@ $\UN^r_n$) is a free module of rank $[L:K]\cdot {r_T \choose r}$ over $\Lambda$ (resp.\@ $\Z_p [\gal{L_n}{L}]$).
\end{thmplain}

While this result shows that non-trivial higher rank universal norms exist, its proof is inherently non-constructive. We shall, however, give an elementary construction of a large $\bLambda$-submodule $\NS^b$ of $\NS^{r_T}$ that is of arithmetic significance as the following result shows.

\begin{thmplain}[Thm.\@ \ref{pairing-theorem}]
    There exists a free rank one $\bLambda$-submodule $\NS^b$ of $\NS^{r_T}$ together with a perfect pairing of $\bLambda$-modules
    \begin{align*}
        \faktor{\NS^{r_T}}{\NS^b} \times \faktor{\bLambda}{\Fitt_{\bLambda}(H^2_{\Sigma,\Iw}(\cO_{L,S}, T))^{**}} \to \faktor{Q(\bLambda)}{\bLambda},
    \end{align*}
    where $Q (\bLambda)$ denotes the total ring of quotients of $\bLambda$ and $H^2_{\Sigma,\Iw}(\cO_{L,S}, T)$ is a modified Iwasawa cohomology group. 
\end{thmplain}

This pairing combines with the cyclotomic equivariant Iwasawa Main Conjecture proven by Burns and Greither \cite{BurnsGreither} to give in Theorem \ref{IMC-refinement} an explicit refinement of the classical cyclotomic Iwasawa main conjecture as follows (see Remark \ref{IMC-refinement-remark} for more details of the precise nature of the relation to the Main Conjecture).

\begin{thmplain}[Thm.\@ \ref{IMC-refinement}]
    Let $K = \QQ$ and let $L$ be the maximal totally real subfield of the cyclotomic field $\Q (\xi_{pf})$ for an integer $f$ coprime to $p$. If $p \nmid [L:\QQ]$, then for every character $\chi$ of $\gal{L}{K}$ 
    there is an isomorphism of $\Lambda_\chi := \ZZ_p(\im(\chi))\llbracket \Gamma \rrbracket$-modules
    \begin{align*}
        \faktor{U^{\infty,\chi}}{\Cyc^{\infty,\chi}} \cong \alpha  \left ( \faktor{\Lambda_\chi}{\cchar_{\Lambda_\chi}(A^{\infty,\chi})} \right),
    \end{align*}
    where $(-)^\chi$ is the functor taking $\chi$-isotypic parts, $U^\infty := \varprojlim_n \cO_{L_n}^\times \otimes_\ZZ \ZZ_p$, $\Cyc^\infty$ is the inverse limit of the groups of $p$-completed cyclotomic units of the field $L_n$, $A^\infty$ is the inverse limit of $p$-parts of the class groups of the fields $L_n$, and $\alpha (-) = \Ext^1_{\Lambda_\chi} ( -, \Lambda_\chi)$ denotes the Iwasawa adjoint. 
\end{thmplain}

It is conjectured, in great generality, that $H^2_\Iw(\cO_{L,S}, T)$ should be finitely generated as a $\Z_p$-module (\textit{c.f.}\@ Conjecture \ref{mu-vanishing-conjecture}). The above pairing allows us to give a reformulation of this conjecture in terms of the quotient $\NS^{r_T}/\NS^b$.

\begin{propplain}[Prop.\@ \ref{mu-vanishing-result}]
    $H^2_\Iw(\cO_{L,S}, T)$ is finitely generated as a $\Z_p$-module if and only if the same is true of $\NS^{r_T}/\NS^b$.
\end{propplain}

Since the aforementioned conjecture is known to hold in several cases one can use this equivalence to obtain several unconditional examples of the finite generation of $\NS^{r_T}/\NS^b$ as a $\Z_p$-module. We give one such example in the setting of elliptic curves in Corollary \ref{mu-vanishing-example}.\medskip \\
We also give applications towards Greenberg's conjecture and equivariant leading term conjectures. For statements of these results the reader is referred to Proposition \ref{GreenbergCriterion} and Theorem \ref{etnc-thm}, respectively.  

\begin{acknowledgments}
    The authors would like to extend their gratitude to David Burns and Takamichi Sano for several stimulating conversations and for their valuable comments on earlier versions of the present manuscript. They would also like to thank Andrew Graham, Martin Hofer and Daniel Macias Castillo for useful comments and discussions, and the referee for their careful reading and helpful suggestions and corrections.\\
    The  authors wish to acknowledge the financial support of the Engineering  and  Physical  Sciences  Research  Council [EP/L015234/1, EP/N509498/1],  the  EPSRC  Centre  for  Doctoral  Training  in  Geometry  and  Number  Theory  (The  London School of Geometry and Number Theory), University College London and King's College London.
\end{acknowledgments}

\markboth{Preliminaries on exterior biduals}{Definition and general properties}
\section{Preliminaries on exterior biduals}
\label{AppendixBiduals}

In this section we both survey the existing theory of exterior biduals and also establish the new results concerning these objects that are needed in 
later sections
of this article. 

\tocless\subsection{Definition and general properties}

Let $R$ be a commutative Noetherian ring, and for any $R$-module put $M^\ast = \Hom_R (M, R)$. 

\begin{definition}
Let $M$ be an $R$-module. 
For any integer $r \geq 0$ we define the  $r$-th \emph{exterior bidual} of $M$ as
\[
\bidual^r_R M = \left ( \exprod^r_R M^\ast \right )^\ast.
\]
In particular, the symbol $\bidual$ does \textit{not} refer to an intersection in this context. 
\end{definition}

This definition is inspired by the notion of \textit{Rubin lattice} introduced in \cite{Rubin96}, and first appeared in the above formalised form in \cite{EulerSystemsSagaI}. See \cite[Rem.~A.9]{EulerSystemsSagaI} for the relation between these two definitions.

\paragraph{Some maps}
Let us now introduce a collection of morphisms that is ubiquitous throughout the theory of exterior biduals. For this purpose, let $M$ and $N$ be $R$-modules. For any integer $r \geq 1$ and morphism $f \in \Hom_R (M, N)$ we define 
\[ 
f^{(r)} \: \exprod^r_R M \to N \otimes_R \exprod^{r - 1}_R M
\]
by
\[
f^{(r)} (m_1 \wedge \dots \wedge m_r) = \sum_{i = 1}^r ( - 1)^{i + 1} f (m_i) \otimes m_1 \wedge \dots \wedge \widehat{m_{i}} \wedge \dots \wedge m_r, 
\]
where we write $\widehat{m_i}$ to mean omission of this particular coefficient. By abuse of notation, we shall simply denote $f^{(r)}$ by $f$ as well. \\

For any integer $s \leq r$ this construction 
(specialised at $N = R$
) 
induces a natural morphism
\begin{equation} \label{wedge-morphism}
\exprod^s_R M^\ast \to \Hom_R \left( \exprod^r_R M, \; \exprod^{r - s}_R M\right), \quad
f_1 \wedge \dots \wedge f_s \mapsto f_s \circ \dots \circ f_1
\end{equation}
which is sometimes referred to as the \textit{rank reduction} by $f$. 
If $\mathfrak{S}_r$ is the permutation group on $r$ elements, then a more explicit description of the above map is given by 
\begin{equation} \label{ExplicitFormula}
f_1 \wedge \dots \wedge f_r \mapsto 
\bigg \{ m_1 \wedge \dots \wedge m_r \mapsto 
\sum_{\sigma \in \mathfrak{S}_{r,s}} \text{sgn} (\sigma) \det( f_i (m_{\sigma (j)}))_{1 \leq, i, j \leq r} m_{\sigma (s +1)} \wedge \dots \wedge m_{\sigma (r)}
\bigg \},
\end{equation}
where $\mathfrak{S}_{r,s} = \{ \sigma \in \mathfrak{S}_r \mid \sigma (1) < \dots < \sigma (s) \text{ and } \sigma (s + 1) < \dots < \sigma (r) \}$. By virtue of this map, we shall regard any element of $\exprod^s_R M^\ast$ as an element of $\Hom_R (\exprod^r_R M, \; \exprod^{r - s}_R M^\ast)$. 

\paragraph{General properties}

Let $N$ and $M$ be $R$-modules, and $r \geq 0$ an integer. 
The following properties are immediate from the properties of the functors $\exprod^r_R -$ and $\Hom_R ( -, R)$:
\begin{itemize}
\item $\bidual^0_R M = R$ and $\bidual^1_R M = M^{\ast \ast}$.
\item Every morphism of $R$-modules $N \to M$ induces a morphism $\bidual^r_R N \to \bidual^r_R M$.
\item There is a natural morphism
\[
\xi^r_M \: \exprod^r_R M \to \bidual^r_R M, \quad m \mapsto \{ f \mapsto f (m) \}
\]
which is neither injective nor surjective in general. If $M$ is a finitely generated projective $R$-module, however, then the map $\xi^r_M$ is an isomorphism (see \cite[Lem. A.1]{EulerSystemsSagaI}).\\
For any $s \leq r$ and $f \in \exprod^s_R M^\ast$ the morphism $\xi^r_M$ fits into a commutative diagram
\begin{cdiagram}
\exprod^r_R M \arrow{r}{f} \arrow{d}{\xi^r_M} & \exprod^{r - s}_R M \arrow{d}{\xi^{r - s}_M} \\
\bidual^r_R M \arrow{r}{f} & \bidual^{r - s}_R M,
\end{cdiagram}
where the bottom map is defined as the dual of 
\[
\exprod^{r - s}_R M^\ast \to \exprod^r_R M^\ast, \quad g \mapsto f \wedge g. 
\]
\item Let $Q$ be the total ring of fractions of $R$ and assume that $R$ is reduced. It is shown in \cite[Prop.\@ A.8]{EulerSystemsSagaI} that the map $\xi^r_M$ induces an isomorphism
\begin{equation} \label{xi-map}
\big \{ a \in Q \otimes_R \exprod^r_R M \mid f (a) \in R \text{ for all } f \in \exprod^r_R M^\ast \big \} \stackrel{\simeq}{\longrightarrow} \bidual^r_R M. 
\end{equation} 
\end{itemize}

\tocless\subsection{Functoriality aspects}

We follow Sakamoto \cite[Appendix B3]{Sakamoto20} in considering the following two conditions on the commutative Noetherian ring $R$:
\begin{itemize}
\item[($G_n$)] The ring $R$ is said to satisfy the condition $(G_n)$ if the localisation $G_\p$ is Gorenstein for any prime $\p \in \Spec R$ of height $\text{ht} (\p) \leq n$. 
\item[($S_n$)] The ring $R$ is said to satisfy Serre's condition $(S_n)$ if $\text{depth } R_\p \geq \text{inf} (n, \text{ht} (\p) )$ holds for any prime ideal $\p \in \Spec R$. 
\end{itemize}

\begin{rk}
For example, $R$ is Cohen-Macaulay if and only if $R$ satsfies $(S_n)$ for all $n \in \N_0$. In particular, Gorenstein rings satisfiy both $(G_n)$ and $(S_n)$ for all $n \in \N_0$. Examples of Gorenstein rings include equivariant Iwasawa algebras of the form considered in the present article (see Lemma \ref{LambdaGorenstein}) as well as finite group rings (see \cite[p.\@ 779]{CurtisReiner}).\\
In the sequel rings satisfying $(G_1)$ and $(S_2)$ will play an important role. Such rings are studied in, for example, \cite{Vasconcelos68} and \cite{Vasconcelos70} where they are referred to as \textit{quasi-normal}.  
\end{rk}

\begin{lem} \label{RyotarosLemma1}
Suppose that $R$ satisfies both $(G_0)$ and $(S_1)$. If $N \hookrightarrow M$ is an injective morphism of $R$-modules, then for any integer $r \geq 1$ the induced map
\[
\bidual^r_R N \to \bidual^r_R M
\] 
is injective as well. 
\end{lem}

\begin{proof} See \cite[Lem.\@ C.1]{Sakamoto20}.
\end{proof}

\begin{lem} \label{BidualsLimits}
Suppose that $R$ satisfies $(S_2)$ and $(G_1)$ and admits a presentation as an inverse limit $R = \varprojlim_{i \in I} R_i$ where each $R_i$ is a Noetherian ring satisfying $(S_2)$ and $(G_1)$.
Let $C^\bullet$ be a complex of projective $R$-modules such that there is an integer $n \geq 0$ satsifying $C^i = 0$ if $i < 0$ or $i > n$. 
 Then there is an isomorphism
\[
\bidual^r_R H^0 (C^\bullet) \stackrel{\simeq}{\longrightarrow} \varprojlim_{i \in I} \bidual^r_{R_i} H^0 (C^\bullet \otimes^\mathbb{L}_R R_n ).
\]
\end{lem}

\begin{proofbox}
See \cite[Lem.\@ B.15]{Sakamoto20}.
\end{proofbox}

\begin{lem} \label{BaseChangeLem}
Let $R$ be a Noetherian ring, and let $R \to R'$ be a ring morphism that endows $R'$ with the structure of an $R$-module of projective dimension at most one.
\begin{enumerate}[label=(\alph*)]
    \item{If $M$ is a finitely generated $R$-module such that $\Ext^1_R (M, R) = 0$ and
    \begin{cdiagram}
   0 \arrow{r} & P_2 \arrow[r, "f"] & P_1 \arrow{r} & R' \arrow{r} & 0,
\end{cdiagram}
is any exact sequence of $R$-modules, where $P_1$ and $P_2$ are finitely generated projective,
then there is a natural exact sequence
\begin{cdiagram}[column sep=tiny]
0 \arrow{r} & \left ( \bidual^r_R M \right ) \otimes_R R' \arrow{r} & \bidual^r_{R'} ( M \otimes_R R')
\arrow{r} & \Ext^1_R \big ( \exprod^r_R M^\ast, P_2 \big) \arrow[r, "f"] & \Ext^1_R \big ( \exprod^r_R M^\ast, P_1 \big).
\end{cdiagram}}
    \item{If $P$ is a finitely generated projective $R$-module and $M \subseteq P$ then we have a commutative diagram
        \begin{cdiagram}
            \left ( \bidual^r_R M \right ) \otimes_R R' \arrow[r, hookrightarrow] \arrow[d] &\bidual^r_{R'} ( M \otimes_R R') \arrow[d]\\
            \left ( \exprod^r_R P \right ) \otimes_R R' \arrow[r, "\sim"] &\exprod^r_{R'} ( P \otimes_R R')
        \end{cdiagram}
        where the vertical arrows are induced by the identification of biduals with exterior powers for finitely generated projective modules and the bottom arrow is the natural base-change isomorphism.
    }
\end{enumerate}

\end{lem}

\begin{rk}
Unlike exterior powers, exterior biduals are not in general compatible with base change (see \cite[Rem.\@ B.9]{Sakamoto20} for a short discussion). Lemma \ref{BaseChangeLem} is therefore an exceptional phenomenon.
\end{rk}

\textit{Proof of \ref{BaseChangeLem}}:
At the outset we fix a finitely generated $R$-module $N$ and note that we have a commutative diagram
\begin{cdiagram}
   & \Hom_R (N, R) \otimes_R P_2 \arrow{d}{\simeq} \arrow{r} & \Hom_R (N, R) \otimes_R P_1 \arrow{d}{\simeq} \arrow{r} & \Hom_R (N, R) \otimes_R R' \arrow{d} \arrow[r] &0\\
   0 \arrow{r} & \Hom_R (N, P_2) \arrow{r} & \Hom_R (N, P_1) \arrow{r} & \Hom_R (N, R')  &.
\end{cdiagram}
Now, the Snake Lemma implies that we have an exact sequence
\begin{equation} \label{snake-lemma}
    \begin{tikzcd}[column sep=small]
        0 \arrow{r} & \Hom_R (N, R) \otimes_R R' \arrow{r} &  \Hom_R (N, R') 
        \arrow{r} & \Ext^1_R (N, P_2) \arrow{r} & \Ext^1_R (N, P_1).
    \end{tikzcd}
\end{equation}
Since $P_2$ is a direct summand of a free $R$-module, the assumption $\Ext^1_R(M,R) = 0$ ensures that 
$\Ext^1_R (M, P_2) = 0$. As a consequence, if we take $N = M$ in the above exact sequence then we deduce that the natural map
\begin{cdiagram}
 \Hom_R (M, R) \otimes_R R' \to \Hom_R (M, R') .
\end{cdiagram}\label{hom-commutes-with-base-change}
is an isomorphism. Taking $R'$-exterior powers and $R'$-duals, we deduce an isomorphism
\begin{align*}
\Hom_{R'} \left( \left(\exprod^r_{R} M^\ast \right) \otimes_R R', \; R'\right) & = \Hom_{R'} \left( \exprod^r_{R'} M^\ast \otimes_R R', \; R'\right) \\
& \cong \Hom_{R'} \left( \exprod^r_{R'} \Hom_R (M, R'), \; R'\right).
\end{align*}
The tensor-hom adjunction implies that
\begin{equation} \label{adjunction}
    \Hom_R (M, R') = \Hom_R (M, \Hom_{R'} (R', R')) = \Hom_{R'} ( M \otimes_R R', R'),
\end{equation}
so we in fact have an isomorphism
\[
\Hom_{R'} \left( \left(\exprod^r_{R} M^\ast \right) \otimes_R R', \; R'\right) \cong \bidual^r_{R'} ( M \otimes_R R').
\]
Using (\ref{adjunction}) again, this time for the left hand side, we can restate this as 
\[
\Hom_{R} \left( \exprod^r_{R} M^\ast, \; R'\right)\cong \bidual^r_{R'} ( M \otimes_R R').
\]
By taking $N = \exprod_R^r M^*$ in (\ref{snake-lemma}) and noting that
\[
\left(\bidual^r_R M \right) \otimes_R R' = \Hom_R \left( \exprod^r_R M^\ast, \; R\right) \otimes_R R'
\]
we then obtain an exact sequence of the desired form. The commutativity of the diagram in the second claim of the Lemma now follows via a straightforward, but somewhat tedious, diagram chase using the naturality of the base-change and tensor-hom isomorphisms.
\qed

\tocless\subsection{Further properties}

\begin{lem} \label{LittleLemma}
Suppose that $R$ satisfies $(G_1)$ and $(S_2)$.
Let $s \geq 1$ be an integer and
\begin{equation} \label{statement-sequence}
\begin{tikzcd}[column sep=3.5em]
   0 \arrow{r} & X \arrow{r} & Y \arrow{r}{\bigoplus_{i = 1}^s \varphi_i} & R^{\oplus s} \arrow{r} & Z \arrow{r} & 0
\end{tikzcd}
\end{equation}
an exact sequence of $R$-modules, where $Y$ is a free $R$-module of rank $d$. Fix an integer $r$ such that $1 \leq s \leq r \leq d$ and consider the map 
\[
\varphi = \bigwedge_{1 \leq i \leq s} \varphi_i \: \exprod^r_R Y \to \exprod^{r - s}_R Y. 
\]
Then the following hold:
\begin{liste}
\item There exists an exact sequence
\begin{cdiagram}
0 \arrow{r} & \bidual^r_R X \arrow{r} & \exprod^r_R Y \arrow{r}{\oplus_{i = 1}^s \varphi_i} & \displaystyle \bigoplus_{i = 1}^s \exprod^{r - 1}_R Y .
\end{cdiagram}
\item We have an inclusion
\[
\im \varphi \subseteq \bidual^{r - s}_R X,
\]
where we regard $\bidual^{r - s}_R X$ as a submodule of $\bidual^{r-s}_R Y = \exprod^{r - s}_R Y$ via Lemma \ref{RyotarosLemma1}.
\item We have an inclusion 
\[
\Fitt^{0}_R (Z) \subseteq \left\{ f (a) \mid a \in \im \varphi, \; f \in \exprod^{r - s}_R X^\ast \right\}  \]
with pseudo-null cokernel. 
\item Take $r = d$ and let $a \in \im \varphi$ be an $R$-generator. If $R$ is reduced and $Z$ is torsion, then $a$ is not a zero-divisor and there is an isomorphism
\[
\Big (\bidual^{d - s}_R X \Big)^\ast \stackrel{\simeq}{\longrightarrow} \Fitt_R^0 (Z)^{\ast \ast}, 
\quad \Phi \mapsto \Phi (a).
\]
\end{liste}
\end{lem}

\begin{proofbox}
Part (a) is a result of Sakamoto \cite[Lem.\@ B.12]{Sakamoto20}. If $r = s$, part (b) is clear and we may therefore assume that $r - s \geq 1$.
The exact sequence from (a) implies that (b) will follow once we have proved that for every $k \in \{1, \dots, s \}$ the composition
\[
\exprod^r_R Y \stackrel{\varphi}{\longrightarrow} \exprod^{r - s}_R Y \stackrel{\varphi_k}{\longrightarrow} \exprod^{r - s - 1}_R Y 
\]
is zero. It now suffices to observe that this composition agrees with the map $\varphi_k \wedge \varphi$ and therefore vanishes because (\ref{wedge-morphism}) is a morphism. \medskip \\
 As for (c), the proof of \cite[Prop.\@ A.2\,(ii)]{EulerSystemsSagaI}
shows that
\begin{equation} \label{evaluation-ideal}
\left\{ f (a) \mid a \in \im \varphi, \; f \in \exprod^{r - s}_R Y^\ast \right\} = \Fitt^{0}_R (Z).
\end{equation}
The inclusion $\bidual^{r - s}_R X \hookrightarrow \exprod^{r - s}_R Y$ induces a restriction map $
\big ( \exprod^{r - s}_R Y \big)^\ast \to \big ( \bidual^{r - s}_R X \big)^\ast$ that fits into the commutative diagram
\begin{cdiagram}
\exprod^{r - s}_R Y^\ast \arrow{r} \arrow{d}{\simeq} & \exprod^{r - s}_R X^\ast \arrow{d} \\
\big ( \exprod^{r - s}_R Y \big)^\ast \arrow{r} & \big ( \bidual^{r - s}_R X \big)^\ast,
\end{cdiagram}
where the top arrow is also induced by restriction and the vertical arrows are the respective canonical maps of the form (\ref{wedge-morphism}). 
If $f \in \exprod^{r - s}_R Y^\ast$ and $a \in \im \varphi$, then, since $\im \varphi \subseteq \bidual^{r - s}_R X$, the above diagram implies that the value $f (a)$ coincides with $a$ evaluated at the image of $f$ in $\exprod^{r - s}_R X^\ast$.\\
We remind the reader that for any $x \in \bidual^{r - s}_R X$ and $f \in \exprod^{r - s}_R X^\ast$, the value $f(x) \in R$ is given by $x (f)$, where $x$ is regarded, true to its definition, as a map $\exprod^{r - s}_R X^\ast \to R$. The above discussion therefore shows that the set (\ref{evaluation-ideal}) is contained 
in $\{ f (a) \mid a \in \im \varphi, \; f \in \exprod^{r - s}_R X^\ast \}$.\\
To show that this inclusion is a pseudo-isomorphism we claim that the cokernel of the morphism
\[
\exprod^{r -s}_R Y^\ast \to \exprod^{r - s}_R X^\ast
\]
is pseudo-null. 
It suffices to show that $\coker \{ Y^\ast \to X^\ast \} = \Ext^1_R (\im  \textstyle{\bigoplus_{i = 1}^s} \varphi_i, R)$ is pseudo-null. To do this, fix a prime ideal $\p \in \Spec(R)$ of height $\text{ht } \p \leq 1$. Then 
\[ 
 \Ext^1_R \big (\im  \textstyle{\bigoplus_{i = 1}^s} \varphi_i, \; R \big)_\p = \Ext^2_{R_\p} (Z_\fp, R_\p) = 0 
\] 
since, by assumption, $R$ satisfies ($G_1$) whence $R_\p$ has injective dimension 1.
Consequently, any $f \in \exprod^{r - s}_{R_\p} X_\p^\ast$ can be lifted to an element $g \in \exprod^{r - s}_{R_\p} Y^\ast_\p$. This shows that
\[
\{ f (a) \mid a \in \im \varphi, \; f \in \exprod^{r - s}_R X^\ast \}_\p = 
\{ f (a) \mid a \in (\im \varphi)_\p, \; f \in \exprod^{r - s}_R X_\p^\ast \}
\]
is contained in the localisation of the set (\ref{evaluation-ideal}) at $\p$. \medskip \\
To prove (d), we note that the assumption that $R$ be reduced implies that the the total ring of fractions $Q$ identifies with a product of fields and is therefore a semi-simple, semi-local ring. Since $Z$ is torsion, we get a short exact sequence
\begin{cdiagram}
0 \arrow{r} & Q \otimes_R X \arrow{r} & Q \otimes_R Y \arrow{r} & Q^{\oplus s} \arrow{r} & 0
\end{cdiagram}
and this implies that $Q \otimes_R X$ is a projective $Q$-module of constant rank $d - s$, hence is $Q$-free of rank $d - s$. It follows that
\begin{equation} \label{bidual-over-Q-free}
Q \otimes_R \Big ( \bidual^{d - s}_R X \Big) \cong \exprod^{d - s}_Q ( Q \otimes_R X) \cong Q 
\end{equation}
is $Q$-free of rank one. As we are taking $r =d$, the module $\im \varphi$ is the image of a free $R$-module of rank one, hence generated over $R$ by an element $a \in \im \varphi$, say. If $a$ was a zero-divisor, then the set (\ref{evaluation-ideal}) would consist of zero-divisors. However, we have
\[
Q \otimes_R \Fitt_R^0 (Z) = \Fitt^0_Q ( Q \otimes_R Z ) = Q
\]
and so this cannot be the case. It follows that $a$ generates a free $R$-submodule of rank one of $\bidual^{d - s}_R X$ and this combines with (\ref{bidual-over-Q-free}) to imply that the quotient $\faktor{\big ( \bidual^{d - s}_R X \big )}{R \cdot a}$ is torsion. We deduce that the map
\[
\Big ( \bidual^{d - s}_R X \Big)^\ast \to (R \cdot a)^\ast \cong R
\]
is injective. In other words, the map
\[
\mathrm{ev} \: \Big ( \bidual^{d - s}_R X \Big)^\ast \to R, \quad \Phi \mapsto \Phi (a)
\]
is injective. In particular, the image of $\mathrm{ev}$ is a reflexive ideal of $R$ and therefore uniquely determined by its localisations at height-one primes of $R$ (see \cite[Lem.\@ C.13]{Sakamoto20}). Given this, part (d) will follow from (c) once we have demonstrated that the natural map
\begin{equation} \label{surjectivity-display}
\exprod^{d - s}_R X^\ast \to \Big ( \bidual^{d - r}_R X \Big)^\ast, \quad
x \mapsto \{ f \mapsto f (x) \} 
\end{equation}
has pseudo-null cokernel. To this end, let $\p \subseteq R$ be a prime of height at most one and $M$ a finitely generated $R_\p$-module. If $M_\tor$ denotes the $R_\p$-torsion submodule, then we have a commutative diagram
\begin{cdiagram}
0 \arrow{r} & M_\tor \arrow{d} \arrow{r} & M \arrow{r} \arrow{d} & \faktor{M}{M_\tor} \arrow{d} \arrow{r} & 0 \\
& 0 \arrow{r} & M^{\ast \ast} \arrow{r} &  \big ( \faktor{M}{M_\tor} \big)^{\ast \ast}.
\end{cdiagram}
The claim therefore follows by combining this diagram with the the fact that over $R_\fp$, which is a reduced Gorenstein ring of Krull dimension one, any finitely generated torsion-free module is reflexive (see \cite[Thm. (6.2)\,(4)]{Bass}, note also the comparison of the notions \textit{torsion-free} and \textit{torsion-less} in \cite[Thm.\@ (A.1)]{Vasconcelos68}).  
\end{proofbox}

\markboth{Higher-rank universal norms}{The set-up}
\section{Higher-rank universal norms}

\subsection{The set-up}\label{set-up-section}\pagestyle{default}
Fix an odd prime $p$ and let $K$ be a number field with $G_K$ its absolute Galois group. We write $S_\infty(K)$ for the set of archimedean places of $K$, and $S_p(K)$ for the set of $p$-adic places of $K$.
Given a Galois extension $F| K$ we write $S_\ram(F|K)$ for the places of $K$ that ramify in $F$ and $S_\mathrm{split}(F|K)$ for the places of $K$ that split completely in $F$. 
If $S$ is a set of places of $K$, we denote by $S(F)$ the union $S \cup S_\ram(F/K)$ and $S_F$ the set of places of $F$ that lie above those contained in $S$. We will, however, often omit the explicit reference to these fields in case it is clear from the context. For example, $\bigO_{F, S}$ shall denote the ring of $S(F)_F$-integers of $F$. 
\medskip \\
Given any commutative unital ring $R$ we write $Q(R)$ for the total quotient ring of $R$; that is to say, the localisation of $R$ at the multiplicative set of non-zero-divisors. If $M$ is an $R$-module, then we denote $M^\lor := \Hom_{R}(M, Q(R)/R)$.

For an abelian group $A$ we denote by $A_\tor$ its torsion-subgroup and by $A_\tf = \faktor{A}{A_\tor}$ its torsion-free part.
If $A$ is finite, we denote by $\widehat{A} = \Hom_\Z (A, \C^\times)$ its character group, and for any $\chi \in \widehat{A}$ we let
\[
e_\chi = \frac{1}{|A|} \sum_{\sigma \in A} \chi (\sigma) \sigma^{-1} \quad \in \C [A]
\]
be the usual primitive orthogonal idempotent associated to $\chi$.
\\

Let $\mathcal{Q}$ be a finite extension of $\Q_p$ with ring of integers $\mathcal{R}$. We also establish the following objects and notations: 
\begin{itemize}
\item $L | K$ a finite abelian extension of number fields with Galois group $\cG$ in which every archimedean place splits completely,
\item $L_\infty | L$ a $\Z_p$-extension in which no non-archimedean place splits completely and such that the extension $L_\infty | K$ is Galois and has Galois group $\Gamma \times \cG$, where $\Gamma = \gal{L_\infty}{L} \cong \Z_p$, 
\item $\Gamma^n = \gal{L_\infty}{L_n}$ the unique subgroup of $\Gamma$ of index $p^n$, and $\Gamma_n = \faktor{\Gamma}{\Gamma^n}$, 
\item $\cG_n = \gal{L_n}{K}$,
\item $\Lambda = \mathcal{R} \llbracket \Gamma \rrbracket = \varprojlim_n \cR [\Gamma_n]$ the Iwasawa algebra and $\bLambda = \mathcal{R} \llbracket \gal{L_\infty}{K} \rrbracket = \varprojlim_n \cR [\cG_n]$ its equivariant counterpart. Due to our assumptions, we have a decomposition $\bLambda = \Lambda [\cG]$. 
\end{itemize}

We now fix a $p$-adic representation $T$ with coefficients in $\cR$. That is to say, a free $\cR$-module endowed with a continuous $G_K$-action that we regard as a sheaf on the \'etale site of $\Spec K$. Assume that the set $S_{\ram}(T)$ of places of $K$ at which $T$ has bad reduction is finite. We then fix a finite set $S$ of places of $K$ containing
\begin{align*}
    S_\infty(K) \cup S_p(K) \cup S_\ram(T).
\end{align*}
Let $T^\ast (1) := \Hom_\cR (T, \cR(1))$, where $\cR (1) = \cR \otimes_{\Z_p} \Z_p (1)$. 

Given a $\bLambda$-module $M$ we write $M^\#$ for the $\bLambda$-module which has the same underlying $\Lambda$-module structure as $M$ but with the $\cG$-action twisted by the involution $g \mapsto g^{-1}$ for $g \in \cG$. Similarly, if $\gamma$ is a topological generator of $\Gamma$, then we write $M^\circ$ for the $\bLambda$-module which has the same underlying $\cR[\cG]$-module structure as $M$ but with the $\Gamma$-action twisted by the involution $\gamma \mapsto \gamma^{-1}$.

\paragraph{$\Sigma$-modified \'etale cohomology}
In this article it is necessary to slightly modify the usual compactly supported \'etale cohomology complex of the representation $T$ in order to ensure that the cohomology in the lowest degree is $\cR$-torsion free. This should, however, be regarded as a convenient technical device rather than an integral feature of the theory since in many interesting cases it can actually be disregarded (see the Examples \ref{RepExamples}). We first briefly recall the definitions of the relevant complexes from \cite[\S 2.3]{EulerSystemsSagaI}.  \\

Let $\Sigma$ be a finite set of places of $K$ that is disjoint from $S(L_0)$. For any place $w \in \Sigma_{L_n}$ denote by $\kappa_w$ the residue field of $\bigO_{L_n, S}$ at $w$. Then we define the \emph{$\bm{\Sigma}$-modified \'etale cohomology complex} of $T$ to be
\[
\text{R} \Gamma_\Sigma (\bigO_{L_n, S}, T) := \cone \Big \{ \text{R} \Gamma_{\et} ( \bigO_{L_n, S}, T) \to \bigoplus_{w \in \Sigma_{L_n}} \text{R} \Gamma_{\et} ( \kappa_w, T) \Big \} [-1]
\]
and set $H^i_\Sigma (\bigO_{L_n, S(L_n)}, T) := H^i (\text{R} \Gamma_\Sigma (\bigO_{L_n, S(L_n)}, T))$ for all $i \in \Z$. 
 We shall fix such a choice of $\Sigma$ for the remainder of this article. For any $i \in \Z$, the \emph{$\bm{\Sigma}$-modified Iwasawa cohomology} of $T$ with respect to the $\Z_p$-extension $L_\infty$ is defined as 
\[
H^i_{\Sigma, \text{Iw}} (\bigO_{L, S}, \; T) = \varprojlim_{n \in \N} H^i_\Sigma(\cO_{L_n, S}, T),
\]
and this limit can be naturally endowed with the structure of a $\bLambda$-module. \\

Throughout this article we suppose, unless explcitly stated otherwise, that the tuple $(T,L_\infty, \Sigma)$ satisfies the following mild hypotheses:
\begin{hypotheses}\label{main-hypothesis}\text{}
    \begin{enumerate}[label=(\arabic*)]
        \item{For every $n \in \NN$ one has that the module of invariants $H^0_\Sigma(L_n, T)$ vanishes.}
        \item{The $\cR$-free module $Y_K(T) = \bigoplus_{v \in S_\infty(K)} H^0(K_v,T^*(1))$ has non-zero rank $r_T$ (which we may often refer to as the \emph{basic rank} of the tuple $(T, L_\infty, \Sigma)$).}
        \item{$H^1_\Sigma(\cO_{L_n, S}, T)$ is $\cR$-torsion-free for every $n \in \N_0$}.
        \item{$H^2_{\Sigma,\Iw}(\cO_{L,S},T)$ is a torsion $\Lambda$-module.}
    \end{enumerate}
\end{hypotheses}

\begin{remark} \label{HypothesesRemark}\phantom{empty}
Hypothesis \ref{main-hypothesis}(4) is (the $\Sigma$-modified version of) the weak Leopoldt conjecture for $p$-adic representations due to Perrin-Riou \cite[\S\,1.3]{PR95}. 
        In fact, in Lemma \ref{mu-vanishing-independent-lemma} below we show that, under additional mild conditions on the tuple $(T,L_\infty, \Sigma)$, this hypothesis is independent of the choice of $\Sigma$ and is thus equivalent to requiring that $H^2 ( \gal{L^S}{L_\infty}, T \otimes_{\cR} \faktor{\cQ}{\cR} (1)) = 0$, where $L^S$ is the maximal Galois extension of $L$ unramified outside $S$ (see, for example, \cite[Prop.\@ 1.3.2]{PR95})).
        The conjecture is known in many cases naturally arising in arithmetic (see \cite[Appendix B]{PR95}) and we shall recall some of these examples below. 
\end{remark}

\begin{bspe1} \label{RepExamples} \phantom{empty}
\begin{liste}
\item     Let $\cR = \ZZ_p$ and $T = \ZZ_p(1)$, then $T$ always satisfies the hypotheses \ref{main-hypothesis}\,(1) and (2). Moreover, for each $n \geq 1$, Kummer theory gives a canonical identification
    \begin{align*}
        H^1_\Sigma(\cO_{L_n, S}, \Z_p (1)) = \Z_p \otimes_\Z \ker \Big \{ 
         \bigO_{L_n, S }^\times \to \bigoplus_{w \in \Sigma_{L_n}} \Big ( \faktor{\bigO_{L_n, S}}{ w}  \Big)^\times
        \Big \}.
    \end{align*}
    The group on the right is the $p$-completion of the $(S , \Sigma)$-unit group $\bigO_{L, S, \Sigma}^\times$ and plays an important role in the context of the Stark conjectures. 
    In particular, if $L$ and $K$ are both totally real, then we may take $\Sigma = \varnothing$.\\
    To see the validity of the Leopoldt conjecture, we recall that there is an exact sequence
    \begin{cdiagram}
    0 \arrow{r} & \varprojlim_n A_{S, \Sigma} (L_n) \arrow{r} & 
    H^2_{\Sigma,\Iw}(\cO_{L,S},\Z_p (1)) \arrow{r} & \varprojlim_n X_{L_n, S \setminus S_\infty (K)} \arrow{r} & 0,
    \end{cdiagram}
    where $A_{S, \Sigma} (L_n)$ denotes the $p$-part of the $S$-ray class group mod $\Sigma$, and $X_{L_n, S \setminus S_\infty (K)}$ is the kernel of the augmentation map $\bigoplus_{w \in (S \setminus S_\infty (K))_{L_n}} \Z_p \to \Z_p$. 
    It is well-known that $\varprojlim_n A_{S, \Sigma} (L_n)$ is $\Lambda$-torsion (see, for example, \cite[Prop.\@ (11.1.4)]{NSW}). Moreover, $\varprojlim_n X_{L_n, S \setminus S_\infty (K)}$ is $\Lambda$-torsion because of our assumption that no finite place of $L$ be split completely in $L_\infty$. The aforementioned exact sequence therefore implies that $H^2_{\Sigma,\Iw}(\cO_{L,S},\Z_p (1))$ is $\Lambda$-torsion as well. 
    \item If $T = \text{T}_p E = H^1_\et ( E_{\overline{\Q}}, \Z_p)^\ast$ is the Tate module of an elliptic curve $E$ defined over $K$, then \ref{main-hypothesis}\,(1) holds, and \ref{main-hypothesis}\,(2) holds because $\text{T}_p E$ is an odd representation (due to the Weil pairing).\\
    If $E ( L)$ is $p$-torsion-free, then we may take $\Sigma = \emptyset$. Indeed, an easy exercise in group cohomology shows that in this case $E (L_n)$ is also $p$-torsion-free for all $n \in \N$ (see, for example, \cite[Prop.\@ (1.6.12)]{NSW}). Moreover, since $S$ contains $S_\ram (T)$, the validity of \ref{main-hypothesis}\,(1) implies that $H^0( \bigO_{L_n, S}, T) = 0$.
    In particular, it follows that
    \[
    H^1 (\bigO_{L_n, S}, T)_\tor \cong H^0 ( \bigO_{L_n, S}, \faktor{(\Q_p \otimes_{\Z_p} T)}{T}) \cong E (L_n) [p^\infty] 
    \]
    and so, for each $n \in \NN$, the module $H^1 ( \bigO_{L_n, S}, T)$ is $\ZZ_p$-torsion-free.\\
    Since every elliptic curve defined over $\Q$ is modular, the validity of the weak Leopoldt conjecture \ref{main-hypothesis}\,(4) for $K = \Q$ follows from \cite[Thm.\@ 12.4\,(i)]{kato}. 
    \item Let $f$ be a normalised cuspidal newform of weight $k \geq 2$ and level $N \geq 5$, and take $\cR$ to be a finite extension of $\Z_p$ that contains the Fourier coefficients of $f$ (using some fixed embedding $\overline{\Q_p} \hookrightarrow \C$). Then one can attach a rational $p$-adic representation $V_f$ of $G_\Q$ to $f$, 
    that is to say a finite dimensional $\mathcal{Q}$-vector space with a continuous $G_\Q$-action,
    see for example \cite{Deligne}. Let $T_f \subseteq V_f$ be a Galois-stable lattice. 
    Since the complex absolute values of the eigenvalues of $\Frob_\q$ for $\q \nmid p N$ are $p^{(k - 1) / 2}$, the representation $T_f$ satisfies hypothesis \ref{main-hypothesis}\,(1). Moreover, the representation $T_f$ is odd, so we have $H^0 (\R, V_f^\ast (1)) \neq 0$ and $T_f$ also satisfies hypothesis \ref{main-hypothesis}\,(2). Finally, $T_f$ satisifies hypothesis \ref{main-hypothesis}\,(4) for $K = \Q$ by \cite[Thm.\@ 12.4\,(i)]{kato}. 
    \end{liste}
\end{bspe1}

Given these definitions, we have the \emph{$\bm{\Sigma}$-modified compactly supported \'etale cohomology complex} 
\[
\text{R}\Gamma_{c, \Sigma} ( \bigO_{L_n, S}, T) = 
\text{R}\Hom_\cR ( \text{R} \Gamma_\Sigma ( \bigO_{L_n, S}, T^\ast (1)), \cR) [-3] \oplus \Big ( \bigoplus_{w \in S_\infty (L_n)} H^0 ( (L_{n})_w, T) \Big ) [-1]
\]
as well as the complex
\begin{align*}
    C_n^\bullet := \text{R}\Hom_\cR(\text{R}\Gamma_{c,\Sigma}(\cO_{L_n, S}, T^*(1)), \cR)[-2].
\end{align*}
Below we record the properties of these constructions that are needed in this article. 

\begin{prop}[\cite{EulerSystemsSagaI}, Prop.\@ 2.22] \label{FiniteComplex}\text{}
\begin{liste}
    \item{$C^\bullet_n$ is acyclic outside degrees zero and one, and is perfect as an element of the derived category $D(\cR [\cG_n])$.}
    \item{There is a canonical isomorphism
        \begin{align*}
            H^0(C^\bullet_n) \cong H^1_{\Sigma}(\cO_{L_n, S}, T)
        \end{align*}
        and a split short exact sequence
        \begin{equation} \label{yoneda-extension-sequence-H2-finite}
        \begin{tikzcd}
           0 \arrow{r} & H^2_\Sigma ( \bigO_{L_n, S}, T) \arrow{r} & H^1 (C^\bullet_n) \arrow{r} & Y_K (T)^\ast \otimes_\cR \cR [\cG_n] \arrow{r} & 0
        \end{tikzcd}
        \end{equation}
        in which the first map is canonical and the second depends on the choice of a set of representatives of the orbits of $\gal{L_n}{K}$ on $S_\infty(L_n)$.
    }
\end{liste}
\end{prop}

Next we introduce the Iwasawa-theoretic variants of the above constructions.\medskip \\ 
Denote by $C^\bullet_\infty$ the complex of $\bLambda$-modules $R\varprojlim_n C^\bullet_n$ taken with respect to the natural codescent morphisms $\phi_n \: C^\bullet_n \to C^\bullet_{n-1}$ (see, for example, \cite[Lem.\@ 7.1\,(v)]{BuSaNC}). Here $R\varprojlim_n$ refers to the homotopy limit taken in the triangulated category $D(\bLambda)$ (which is unique up to non-unique isomorphism) and is defined so as to fit in an exact triangle
\begin{equation}\label{holim-triangle}
\begin{tikzcd}
    R\varprojlim_n C^\bullet_n \arrow{r} &  \prod_{n} C^\bullet_n \arrow{rr}{(1-\phi_n)_n} & &  \prod_n C^\bullet_n \arrow{r} & \phantom{X}.
\end{tikzcd}
\end{equation}
 We then have the following analogue of Proposition \ref{FiniteComplex}.

\begin{prop}
\begin{liste}
    \item{$C^\bullet_\infty$ is acyclic outside degrees zero and one, and is perfect as an element of the derived category $D(\bLambda)$.}
    \item{There is a canonical isomorphism
        \begin{align*}
            H^0(C^\bullet_\infty) \cong H^1_{\Sigma, \Iw}(\cO_{L,S}, T)
        \end{align*}
        and a split short exact sequence
    \begin{equation} \label{yoneda-extension-sequence-H2}
        \begin{tikzcd}[column sep=small] 
            0 \arrow{r} & H^2_{\Sigma,\Iw}(\cO_{L, S}, T) \arrow{r} & H^1(C^\bullet_\infty) \arrow{r} & \varprojlim_n (Y_K(T)^* \otimes_\cR \cR[\cG_n]) \arrow{r} & 0,
        \end{tikzcd}
        \end{equation}
        where the injection is canonical and the surjection depends on a choice of a set of representatives of the orbits of $\gal{L_\infty}{K}$ on $S_\infty(L_\infty)$.
    }
\end{liste}
\end{prop}

\begin{proof}
    At the outset we remark that an analysis of the long exact sequence of cohomology of the triangle (\ref{holim-triangle}) yields, for each $i \in \ZZ$, an exact sequence
    \begin{cdiagram}
         0 \arrow[r] &\textstyle\varprojlim_n^{(1)} H^{i-1}(C_n^\bullet) \arrow[r] & H^i(C_\infty^\bullet) \arrow[r] & \varprojlim_n H^i(C_n^\bullet) \arrow[r] &0
    \end{cdiagram}
    where the limits are taken with respect to the morphisms induced by $\phi_n$. In particular, since each $H^i(C^\bullet_n)$ is finitely generated as a $\cR[\cG_n]$-module it can be endowed with the structure of a compact Hausdorff space. As such, the left-hand derived limit vanishes. Given this, the first claim of (a) and both claims of (b) follow immediately from Proposition \ref{FiniteComplex}. The fact that $C_\infty^\bullet$ is a perfect complex is proven in \cite[Prop.\@ 1.6.5\,(2)]{fukaya-kato}.
\end{proof}

\begin{lemma}\label{standard-representative-lemma}
    There exists a quadratic standard representative $[\Pi \xrightarrow{\psi} \Pi]$ (in the sense of \cite[Def.\@ A.6]{EulerSystemsSagaI}) of the complex $C_\infty^\bullet$ with respect to the surjection $H^1(C_\infty^\bullet) \xrightarrow{f} Y_K(T)^* \otimes_\cR \bLambda$. 
\end{lemma}

\begin{proof}
    By definition we are required to exhibit a representative $[\Pi \xrightarrow{\psi} \Pi]$ of $C_\infty^\bullet$ in $D(\bLambda)$ with the property that for the free module $\Pi$ there exists a basis $\{b_1,\dots, b_d\}$ of $\Pi$ and an exact sequence
    \begin{align*}
        \langle b_{r_T+1}, \dots, b_d\rangle_{\bLambda} \to H^1(C_\infty^\bullet) \xrightarrow{f} Y_K(T)^* \otimes_\cR \bLambda \to 0
    \end{align*}
    where the first map is induced by the natural map $\Pi \to H^1(C_\infty^\bullet)$. This is proved in
    \cite[Lem.\@ 7.10]{BuSaNC} (where the complex $C_\infty^\bullet$ is denoted $C_{L_\infty, S(L_0)}(T)$ in \textit{loc.\@ cit.\@}).
\end{proof}

Fix a representative $[\Pi \xrightarrow{\psi} \Pi]$ of $C_\infty^\bullet$ where $\Pi$ is a free $\bLambda$-module of rank $d$. Then for any given $n \in \NN_0$, the complex $C_n^\bullet$ is represented by $[\Pi_n \xrightarrow{\psi_n} \Pi_n]$ where we write $\Pi_n := \Pi \otimes_{\bLambda} \cR[\cG_n]$ and similarly for $\psi_n$ (see \cite[Lem.\@ 7.10\,(iii)]{BuSaNC}). In particular, we have exact sequences
\begin{align}\label{yoneda-extension-sequence}
& \begin{tikzcd}[ampersand replacement=\&]
0 \arrow{r} \& H^1_\Sigma(\cO_{L_n, S}, T) \arrow{r}{\phi_n} \& \Pi_n \arrow{r}{\psi_n}  \& \Pi_n
\arrow{r} \& H^1 (C^\bullet_n) \arrow{r} \& 0
.
\end{tikzcd} \\
\label{yoneda-extension-sequence-Iwasawa}
& \begin{tikzcd}[ampersand replacement=\&]
0 \arrow{r} \& H^1_{\Sigma, \Iw}(\cO_{L, S}, T) \arrow{r}{\phi} \& \Pi \arrow{r}{\psi}  \& \Pi
\arrow{r} \& H^1 (C^\bullet_\infty) \arrow{r} \& 0
.
\end{tikzcd}
\end{align}
Let $\{b_1, \dots, b_d \}$ be the $\bLambda$-basis of $\Pi$ chosen in the proof of Lemma \ref{standard-representative-lemma} and, for each $i \in \{1, \dots, d\}$ write $b_i^\ast \in \Pi^\ast$ for the dual of $b_i$. By construction, the $\bLambda$-free module $Y_K (T)^\ast \otimes_\cR \bLambda$ can then be identified with the direct summand $\bigoplus_{i = 1}^{r_T} \bLambda b_i$ of $\Pi$. It follows that we also have the exact sequence
\begin{align}\label{yoneda-extension-sequence-Iwasawa-modified}
& \begin{tikzcd}[ampersand replacement=\&, column sep=small]
0 \arrow{r} \& H^1_{\Sigma, \Iw}(\cO_{L, S}, T) \arrow{r} \& \Pi \arrow{rrr}{\bigoplus_{i = r_T  + 1}^d\psi_i}  \& \& \&  \bLambda^{\oplus (d - r_T)}
\arrow{r} \& H^2_{\Sigma, \Iw} (\bigO_{L, S}, T) \arrow{r} \& 0
,
\end{tikzcd}
\end{align}
where we have set $\psi_i = b_i^\ast \circ \psi$ for each $i \in \{ r_T + 1, \dots, d\}$.

\subsection{The structure of universal norms in higher rank}

In this section we prove a number of basic results about higher rank universal norms. In doing so, we will heavily rely on the notion of \textit{exterior biduals} reviewed in the last section.

\paragraph{Definition of universal norms}

Given integers $m \geq n \geq 1$ we have the corestriction maps
\begin{align*}
    \cores_{m,n} \: H^1_\Sigma(\cO_{L_m, S}, T) \to H^1_\Sigma(\cO_{L_n, S}, T).
\end{align*}
If $r \geq 1$, then these maps induce natural maps on the exterior biduals
\begin{align*}
    \cores_{m,n}^r \: \bidual_{\cR[\cG_m]}^r H^1_\Sigma(\cO_{L_m, S}, T) \to \bidual_{\cR[\cG_n]}^r H^1_\Sigma(\cO_{L_n, S}, T).
\end{align*}
in the following way: First observe that the natural codescent map $\Pi_m \to \Pi_n$ restricts to give $\cores_{m, n}$ on $H^1_\Sigma(\cO_{L_m, S}, T)$. With this in mind we then define $\cores_{m,n}^r$ to be the leftmost vertical arrow in the commutative diagram that is obtained by applying Lemma \ref{LittleLemma}\,(a) to the exact sequences (\ref{yoneda-extension-sequence}) for the levels $m$ and $n$:
\begin{cdiagram}
0 \arrow{r} & \bidual_{\cR[\cG_m]}^r H^1_\Sigma(\cO_{L_m, S}, T) \arrow{r} \arrow[dashed]{d} & 
\exprod^r_{\cR [\cG_m]} \Pi_m \arrow{r} \arrow{d} & \Pi_m \otimes_{\cR [\cG_m]} \exprod^{r - 1}_{\cR [\cG_m]} \Pi_m \arrow{d} \\
0 \arrow{r} & \bidual_{\cR[\cG_n]}^r H^1_\Sigma(\cO_{L_n, S}, T) \arrow{r} & 
\exprod^r_{\cR [\cG_n]} \Pi_n \arrow{r} & \Pi_n \otimes_{\cR [\cG_n]} \exprod^{r - 1}_{\cR [\cG_n]} \Pi_n
\end{cdiagram}
We also refer the reader to \cite[Rem.\@ 6.2 and Rem.\@ 6.11]{EulerSystemsSagaII} for more details on this map.

\begin{definition}\label{UN-definition}
    Fix integers $r \in \N$ and $n \in \NN_0$. 
    \begin{liste}
    \item We define the module of \emph{universal norms} of rank $r$ and level $n$ for $T$ to be
    \begin{align*}
        \UN^r_n = \UN^r_n(T,L_\infty) := \bigcap_{m \geq n} \im(\cores_{m,n}^r)
    \end{align*}
    We remark that $\UN^r_n(T,L_\infty)$ can be naturally regarded as an $\cR[\cG_n]$-module.
    \item We define the module of \emph{norm-coherent sequences} of rank $r$ for $T$ to be
    \begin{align*}
        \NS^r = \NS^r(T,L_\infty) := \varprojlim_{n \in \NN} \bidual_{\cR[\cG_n]}^r H^1_\Sigma(\cO_{L_n, S}, T)
        = \bidual^r_\bLambda H^1_{\Sigma, \text{Iw}} ( \cO_{L, S}, T)
    \end{align*}
    where the inverse limit is taken with respect to the maps $\cores_{m,n}^r$ and we have used Lemma \ref{BidualsLimits} for the last identification. We will often use the notations above interchangeably except in cases of ambiguity. We remark that $\NS^r(T,L_\infty)$ can be naturally regarded as a $\bLambda$-module.
    \end{liste}
\end{definition}

\paragraph{The descent isomorphism}
The following is one of the main results of this article.

\begin{thm} \label{UN-structure-theorem}
Fix an integer $1 \leq r \leq r_T$. 
\begin{liste}
\item The module $\NS^r$ is $\Lambda$-free of rank $[L : K] \cdot \binom{r_T}{r}$.
\item{The natural map
            \begin{align*}
                \NS^r \to \bidual_{\cR[\cG_n]}^r H^1_\Sigma(\cO_{L_n, S}, T)
            \end{align*}
            induces an isomorphism of $\cR[\cG_n]$-modules
            \begin{align*}
                \NS^r \otimes_{\bLambda} \cR[\cG_n] \cong \UN_n^r.
            \end{align*}
            In particular, $\UN_n^{r}$ is a free $\cR[\Gamma_n]$-module of rank $[L:K]\cdot{r_T \choose r}$.
        }
	\item There is an exact sequence
	\begin{cdiagram}[column sep=small]
	0 \arrow{r} & \UN^r_n \arrow{r} & \bidual^r_{\cR [\cG_n]} \UN^1_n \arrow{r} & 
	\Ext^1_\bLambda \Big ( \exprod^r_\bLambda H^1_{\Sigma, \Iw} (\cO_{L, S}, T)^\ast, \bLambda \Big)^{\Gamma^n} \arrow{r} & 0.
	\end{cdiagram}
	The module on the right is a finite $p$-group which vanishes if, for example, $p \nmid [L:K]$.
\end{liste}
\end{thm}

\newpage
\begin{rk} \label{structure-theorem-remark}\text{}
\begin{liste}
\item In the case of $r_T = 1$, universal norms have previously been studied by many authors: The first to obtain a result similar to Theorem \ref{UN-structure-theorem} (for $T = \Z_p (1)$) was Kuz'min \cite{Kuzmin};
later Greither \cite{Greither} also gave a proof in the abstract setting of a system of Galois modules satisfying certain natural axioms. The article \cite{kato2006universal} considers the non-commutative case but also gives an overview of the classical theory that is similar in spirit to our treatment. 
In the setting of elliptic curves a similar result is due to Mazur and Rubin \cite[Theorem 4.2]{MazurRubin03}.
\item Suppose $p \nmid [L:K]$. Then it is well-known that any $\bLambda$-module that is $\Lambda$-projective is necessarily $\bLambda$-projective (see, for example, \cite[Lem.\@ 5.4.16]{NSW}). Since $\bLambda$ is a semi-local ring in this case, and $\NS^{r}$ has constant local rank ${r_T \choose r}$ by the calculation of the proof below, it follows that $\NS^{r}$ is necessarily $\bLambda$-free of rank ${r_T \choose r}$. An analogous statement for universal norms now also follows by codescent. 
\end{liste}
\end{rk}

\textit{Proof of Theorem \ref{UN-structure-theorem}:}
Consider the complex $D^\bullet$ represented by 
\begin{equation}\label{higher-rank-complex}
\begin{tikzcd}
\exprod^r_\bLambda \Pi \arrow{r}{\psi} & \Pi \otimes_\bLambda \exprod^{r - 1}_\bLambda \Pi.
\end{tikzcd}
\end{equation}
By virtue of Lemma \ref{LittleLemma}\,(a) we have that $H^0 ( D^\bullet) = \bidual_{\bLambda}^{r} H^1_{\Sigma, \Iw}(\cO_{L, S}, T)$. In particular, $\left ( \bidual_{\bLambda}^{r} H^1_{ \Sigma, \Iw}(\cO_{L, S}, T) \right)^\Gamma  = 0$ since $\exprod^r_\bLambda \Pi$ is $\Lambda$-free. \\
Moreover, the complex $D^\bullet$ is clearly perfect and the complex $D_n^\bullet = D^\bullet \otimes_\bLambda^\mathbb{L} \cR [\cG_n]$ is represented by 
\begin{cdiagram}
\exprod^r_{\cR [\cG_n]} \Pi_n \arrow{r}{\psi_n} & \Pi_n \otimes_{\cR [\cG_n]} \exprod^{r - 1}_{\cR [\cG_n]} \Pi_n
\end{cdiagram}
which has $H^0 ( D^\bullet_n) = \bidual_{\cR[\cG_n]}^{r} H^1_\Sigma(\cO_{L_n, S}, T)$ as its cohomology in its lowest degree.

Now fix a topological generator $\gamma_n$ of $\Gamma^n$. Then the decomposition $\bLambda = \Lambda[\cG]$ implies that there is an exact sequence
\begin{equation}\label{group-ring-resolution}
\begin{tikzcd}[column sep=3.5em]
    0 \arrow{r} & \bLambda \arrow{r}{\cdot (1-\gamma_n)} \arrow{r} & \bLambda \arrow{r} &  \cR[\cG_n] \to 0 .
\end{tikzcd}
\end{equation}
From this it follows that for any $\bLambda$-module $M$ and $i \geq 2$, the module $\Tor_i^\bLambda(\cR[\cG_n], M)$ vanishes.
Since the complex $D^\bullet$ is acyclic outside degrees zero and one, we then deduce that the spectral sequence
\begin{equation} \label{spectral-sequence}
E^{i , j}_2 = \Tor_{ - i}^\bLambda ( \cR [\cG_n], \; H^j (D^\bullet)) \; \Rightarrow \; E^{i + j} = H^{i + j} ( D^\bullet \otimes_{\bLambda}^\mathbb{L} \cR [\cG_n] ) 
\end{equation}
degenerates on its second page into a collection of short exact sequences. In particular, there is an injection
\begin{equation}\label{bidual-injection}
\left ( \bidual_{\bLambda}^{r} H^1_{\Sigma, \Iw}(\cO_{L, S}, T) \right) \otimes_\bLambda \cR [\cG_n] \hookrightarrow \bidual^r_{\cR [\cG_n]} H^1_\Sigma (\bigO_{L_n, S}, T)
\end{equation}
 from which one sees that the coinvariants module $\left ( \bidual_{\bLambda}^{r} H^1_{\Sigma, \Iw}(\cO_{L, S}, T) \right)_\Gamma$ is $\cR$-free. This implies that $\bidual_{\bLambda}^{r} H^1_{\Sigma, \Iw}(\cO_{L, S}, T)$ is $\Lambda$-free (see, for example, \cite[Prop.\@ 5.3.19]{NSW}).\medskip\\ 
To prove Part (a) of Theorem \ref{UN-structure-theorem} it now remains to demonstrate that the $\Lambda$-rank of $\NS^{r}$ is $[L:K]\cdot{r_T \choose r}$. 
Since $Q(\bLambda)$ is semi-simple,
an analysis of the exact sequence (\ref{yoneda-extension-sequence-H2}) implies that
\begin{align*}
\rank_{Q(\bLambda)}(Q(\bLambda) \otimes_{\bLambda} H^1(C^\bullet_\infty)) & = 
    \rank_{Q(\bLambda)}(Q(\bLambda) \otimes_{\bLambda} H^2_{ \Sigma, \Iw}(\cO_{L, S}, T))  \\
    & \phantom{=} + \rank_{Q(\bLambda)}(Q(\bLambda) \otimes_{\bLambda} \varprojlim_{n \in \N} (Y_K(T) \otimes_\cR \cR[\cG_n])).
\end{align*}
By the assumed validity of the weak Leopoldt conjecture we therefore have the equality of ranks
\begin{align*}
    \rank_{Q(\bLambda)}(Q(\bLambda) \otimes_{\bLambda} H^1(C_\infty^\bullet)) &= \rank_{Q(\bLambda)}(Q(\bLambda) \otimes_{\bLambda} \varprojlim_n (Y_K(T) \otimes_\cR \cR[\cG_n])) = r_T.
\end{align*}
On the other hand, an analysis of the Yoneda 2-extension (cf.\@ (\ref{yoneda-extension-sequence-Iwasawa}))
\begin{cdiagram}
    0 \arrow{r} & H^1_{\Sigma, \Iw}(\cO_{L,S},T) \arrow{r} & \Pi \arrow{r} &
    \Pi \arrow{r} & H^1(C_\infty^\bullet) \arrow{r} & 0
\end{cdiagram}
associated to the complex $C_\infty^\bullet$ yields the equality $\rank_{Q(\bLambda)}(Q(\bLambda) \otimes_{\bLambda} H^1_{\Sigma,\Iw}(\cO_{L,S}, T)) = r_T$. We may thus calculate
\begin{align*}
    \rank_{Q(\Lambda)}(Q(\Lambda) \otimes_\Lambda \NS^{r}) &= [L:K]\cdot \rank_{Q(\bLambda)}(Q(\bLambda) \otimes_{\bLambda} \NS^{r})\\
    &= [L:K] \cdot \rank_{Q(\bLambda)}\left(Q(\bLambda) \otimes_{\bLambda} \bidual_{\bLambda}^{r} H^1_{\Sigma, \Iw}(\cO_{L, S}, T)\right)\\
    &= [L:K]\cdot \rank_{Q(\bLambda)}\left(\exprod_{Q(\bLambda)}^{r} Q(\bLambda) \otimes_{\bLambda} H^1_{\Sigma, \Iw}(\cO_{L, S}, T)\right)\\
    &= [L:K]\cdot {r_T \choose r}.
\end{align*}

Before we proceed with the proof of Part (b) we first require the following Lemma for which we include the proof for lack of a better reference (and see Remark \ref{alternative-proof} below for another proof in our situation suggested to us by the referee).

\begin{lemma} \label{compactness-argument}
    Let $M = (M_n, \phi_{m,n})$ be an inverse system in the category of compact Hausdorff spaces with limit $M_\infty$. Fix $n \in \NN_0$ and write
    \begin{align*}
        \UN_n(M) := \bigcap_{m \geq n} \im(\phi_{m,n}).
    \end{align*}
    Then the natural map $M_\infty \to M_n$ has image $\UN_n(M)$.
\end{lemma}

\begin{proof}
    Without loss of generality (and for notational simplicity) we prove the statement for $n = 0$.\\
    Suppose to be given an element $u \in \UN_0(M)$. We shall inductively construct an element $m \in M_\infty$ with the property that $m_0 = u$. Suppose that for some $s \geq 1$ we have constructed a coherent tuple $(m_i)_{0 \leq i \leq s}$ such that $m_0 = u$ and each $m_0 \in \UN_i(M)$. For $j \geq s$ define the sets
    \begin{align*}
        X_s = \phi_{s+1,s}^{-1}(m_s), \quad Y_{j,s} = \im(\phi_{j,s+1})
    \end{align*}
    Then both $X_s$ and $Y_{j,s}$ are closed. Indeed, the former is so by the fact that $\phi_{s+1,s}$ is continuous and the latter from the fact that $Y_{j,s}$ is a compact subspace of a Hausdorff space.\\
    It then follows that the intersection $X_s \cap Y_{j,s}$ is also closed and, by hypothesis, non-empty. In particular, the descending filtration $(X_s \cap Y_{j,s})_{j > s}$ satisfies the finite intersection property. Since $M_{s+1}$ is compact we then have that the intersection $\bigcap_{j > s} X_s \cap Y_{j,s}$
    is non-empty. We can therefore take $m_{s+1}$ to be any element of this intersection. Continuing in this fashion we may inductively construct an element $m = (m_i)$ of $M_\infty$ with the desired property. 
\end{proof}

Returning now to the proof of Theorem \ref{UN-structure-theorem}\,(b), we note that the augmentation ideal in $\bLambda$ relative to $\cG_n$ applied to $\NS^r$ is contained in the kernel of the natural map \begin{align*}
    \NS^r \to \bidual_{\cR[\cG_n]}^r H^1_\Sigma(\cO_{L_n, S}, T).
\end{align*}
Given this, we may apply the above Lemma \ref{compactness-argument} to conclude that for all $r \in \NN$ and $n \in \N_0$ the induced map
\begin{align*}
    \iota_n: \left(\varprojlim_n \bidual_{\cR[\cG_n]}^r H^1_\Sigma(\cO_{L_n, S}, T)\right) \otimes_\bLambda \cR[\cG_n] \longrightarrow \bidual_{\cR[\cG_n]}^{r} H^1_\Sigma(\cO_{L_n, S}, T)
\end{align*}
has image $\UN_{n}^{r}$.
On the other hand, the map of (\ref{bidual-injection}) factors through $\iota_n$ via the identification $\bidual^r_\bLambda H^1_{\Sigma, \Iw} ( \bigO_{L, S}, T) \cong \varprojlim_n \bidual_{\cR[\cG_n]}^r H^1_\Sigma(\cO_{L_n, S}, T)$ and so it is in fact an isomorphism. This establishes part (b) of the Theorem.\medskip \\
Finally, the exact sequence in (c) follows directly by combining Lemma \ref{BaseChangeLem} with the explicit resolution (\ref{group-ring-resolution}). To prove that the cokernel appearing in this exact sequence is finite, we claim that the $\cR$-rank of the $\cR$-free module $\bidual^r_{\cR [\cG_n]} \UN^1_n$ agrees with the $\cR$-rank of $\UN^r_n$. 

To do this we first observe that the isomorphism in (b) gives rise to the isomorphism
\begin{align*}
    \cQ \otimes_\cR \UN^1_n & \cong \cQ \otimes_\cR \big ( \cR [\cG_n] \otimes_\bLambda H^1_{\Sigma, \Iw} (\bigO_{L, S}, T) \big) \cong \cQ [\cG_n] \otimes_{Q (\bLambda)} \big ( Q (\bLambda) \otimes_\bLambda H^1_{\Sigma, \Iw} (\bigO_{L, S}, T) \big)
\end{align*}
In particular, since we have already seen that that $Q(\bLambda) \otimes_\bLambda H^1_{\Sigma, \Iw} (\bigO_{L, S}, T)$ is $Q(\bLambda)$-free of rank $r_T$ it follows that $\cQ \otimes_\cR \UN^1_n$ is $\cQ [\cG_n]$-free of rank $r_T$.
We may thus calculate
\begin{align*}
    \mathrm{rk}_{\cR} \big ( \bidual^r_{\cR [\cG_n]} \UN^1_n \big) & = 
    \mathrm{dim}_{\cQ} \big ( \cQ \otimes_\cR \bidual^r_{\cR [\cG_n]} \UN^1_n \big) \\
    & = |\cG_n| \cdot \mathrm{rk}_{\cQ [\cG_n]} \big (  \exprod^r_{\cQ [\cG_n]} \cQ \otimes_\cR \UN^1_n \big) \\
    & = [L_n : K] \cdot \binom{r_T}{r}
\end{align*}
and this matches the $\cR$-rank of $\UN^r_n$ calculated in (b). If $p \nmid | \cG|$, then the observation made in Remark \ref{structure-theorem-remark}(b) implies that $H^1_{\Sigma, \Iw} (\bigO_{L, S}, T)$ is a free $\bLambda$-module. From this it follows that the module $\Ext^1_\bLambda ( \exprod^r_\bLambda H^1_{\Sigma, \Iw} (\bigO_{L, S}, T)^\ast, \bLambda)$ vanishes, as required. \medskip \\
This concludes the proof of Theorem \ref{UN-structure-theorem}. 
\qed

\begin{remark}\label{alternative-proof}
    We are grateful to the referee for pointing out the following easy algebraic proof of Lemma \ref{compactness-argument} in the situation at hand (rather than the more general setting of compact Hausdorff spaces). Namely, suppose that each $M_n$ is a finitely generated $\cR[\cG_n]$-module with each $\phi_{m,n}$ surjective. Then the limit $M_\infty$ is a finitely generated $\bLambda$-module. Consider now the exact sequence $M_\infty \to \UN_n(M) \to Q_n \to 0$. By passing to the limit over $n$, and noting that all modules involved are compact Hausdorff, one deduces that the sequence $M_\infty \to \varprojlim_n \UN_n(M) \to \varprojlim_n Q_n \to 0$ is exact. By definition the first of these maps must be an isomorphism whence $\varprojlim_n Q_n = 0$. Since this an inverse limit with surjective transition maps, indexed over a countable set, one deduces that each $Q_n = 0$ whence the claim.
\end{remark}

\subsection{An Iwasawa-theoretic pairing}

\paragraph{Basic Norms}

At the outset of this section we define the following projection map:
\begin{align*}
    \Theta \: \Det_{\bLambda}(C^\bullet_\infty) &\hookrightarrow Q(\bLambda) \otimes_{\bLambda} \Det_{\bLambda}(C^ \bullet_\infty)\\
    &\xrightarrow{\simeq} \Det_{Q(\bLambda)}(Q(\bLambda) \otimes_{\bLambda} C^\bullet_\infty)\\
    &\xrightarrow{\simeq} Q(\bLambda) \otimes_{\bLambda} \bigg(\exprod_{\bLambda}^{r_T} H^1_{\Sigma, \Iw}(\cO_{L,S}, T) \otimes_{\bLambda} \Big(\exprod_{\bLambda}^{r_T} \varprojlim_n (Y_K(T)^* \otimes_\cR \cR[\cG_n])\Big)^*\bigg)\\
    &\xrightarrow{\simeq} Q(\bLambda) \otimes_{\bLambda} \exprod_{\bLambda}^{r_T} H^1_{\Sigma, \Iw}(\cO_{L,S}, T),
\end{align*}
where the second arrow follows from the base-change property of determinant functors, the third follows by passing to cohomology and noting that the weak Leopoldt conjecture is assumed to hold, and the final arrow follows from collapsing the exterior power with respect to a fixed $\bLambda$-basis of the inverse limit.

We now have the following Lemma which will prove useful in the sequel.

\begin{lemma} 
\label{projection-map-lemma-1}
   The image of $\Theta$ is contained in $\NS^{r_T}$.
\end{lemma}

\begin{proof}
   Recall that in Lemma \ref{standard-representative-lemma} we have chosen a quadratic standard representative $[\Pi \stackrel{\psi}{\longrightarrow} \Pi]$ of the complex $C^\bullet_\infty$ and thereby also fixed a basis $\{ b_1, \dots, b_d \}$ of $\Pi$. 
    By applying \cite[Lem.\@ A.7]{EulerSystemsSagaI} to the complex $Q (\bLambda) \otimes_\bLambda^\mathbb{L} C^\bullet_\infty$  we then see that the projection map $\Theta$ coincides with the rank reduction map
    \[
    \pi_{\psi} = (-1)^{r_T (d - r_T)} \cdot \bigwedge_{r_T < i \leq d} (\psi \circ b_i^\ast) \: \exprod^{d}_\bLambda \Pi \to \exprod^{r_T}_\bLambda \Pi,
    \]
    where each $b_i^\ast \in \Pi^\ast$ denotes the dual of $b_i$.
    By Lemma \ref{LittleLemma}, applied to the exact sequence (\ref{yoneda-extension-sequence-Iwasawa-modified}), the image of $\pi_\psi$
    is contained in $\bidual^{r_T}_\bLambda H^1_{\Sigma, \Iw} (\bigO_{L, S}, T)$.
\end{proof}

\begin{definition}
We define the $\bLambda$-module of \emph{basic norm-coherent sequences} $\NS^b = \NS^b(T, L_\infty)$ to be the image of the homomorphism
    \begin{align*}
       \Theta \: \Det_{\bLambda}(C^\bullet_\infty) \longrightarrow \NS^{r_T}(T, L_\infty) .
    \end{align*} 
\end{definition}

\begin{proposition}\label{basic-theorem}\text{}
 The $\bLambda$-submodule $\NS^b = \NS^b(T,L_\infty)$ of $\NS^{r_T}$ is $\bLambda$-free of rank one. In particular, the quotient $\faktor{\NS^{r_T}}{\NS^b}$ is $\bLambda$-torsion.
\end{proposition}

\begin{proofbox}
It suffices to observe that $\Det_{\bLambda}(C^\bullet_\infty)$, and thus $\NS^b$, is a free $\bLambda$-module of rank one. In particular, both $\NS^{r_T}$ and $\NS^b$ are free $\Lambda$-modules of rank $[L:K]$ and so their quotient is $\bLambda$-torsion.
\end{proofbox}

The following result shows that, at least conjecturally, the elements of $\NS^b(\ZZ_p(1), L_\infty)$ are familiar objects.

\begin{proposition}\label{eimc-proposition}
    Assume that for each $n \in \NN$ the $p$-part of the Rubin-Stark conjecture is valid for the data $(L_n | K, S, \Sigma, S_\infty(K))$ (as is formulated, for example, in \cite[Conj. 2.1]{BKS2}) and let $\varepsilon_n$ be the corresponding Rubin-Stark element. Assume, moreover, that the equivariant Iwasawa Main Conjecture (eIMC) is valid for the data $(L_\infty | K, S, \Sigma, p)$ (as is formulated in \cite[Conj. 3.1]{BKS2}). Then
    \begin{align*}
        \NS^b(\ZZ_p(1), L_\infty) = \langle(\varepsilon_n)_n\rangle_{\bLambda}.
    \end{align*}
\end{proposition}

\begin{proofbox}
    It is well-known that the family $(\varepsilon_n)_n$ constitutes an element of the module\\
    $\NS^{r_T}(\ZZ_p(1), L_\infty)$, see \cite[Prop.\@ 6.1]{Rubin96}. Denote by
    \begin{align*}
        \mathfrak{z}_\infty \in Q(\bLambda) \otimes_{\bLambda} \Det_{\bLambda}(C_\infty^\bullet)
    \end{align*}
    the inverse image of $(\varepsilon_n)_n$ under the latter three arrows in the definition of $\Theta$. Then, after taking into account the equivalent formulation \cite[Conj. 3.7]{BKS2} of the eIMC, one knows that $\mathfrak{z}_\infty$ is a $\bLambda$-basis of $\Det_{\bLambda}(C^\bullet_\infty)$. The Proposition now follows immediately from the definition of the module $\NS^b$.
\end{proofbox}

The following Theorem provides a direct link between the Galois module structures of the quotient appearing in Proposition \ref{basic-theorem} and $H^2_{\Sigma,\Iw}(\cO_{L, S}, T)$. 
We remark that if $I$ is an ideal of $\bLambda$ then one can, in turn, regard its reflexive hull $I^{\ast \ast}$ as an ideal of $\bLambda$.

\begin{thm}\label{pairing-theorem}\text{}
    \begin{liste}
        \item{There exists a canonical isomorphism of $\bLambda$-modules
            \begin{align*}
                \Ext^1_{\bLambda}  \left ( \faktor{\NS^{r_T}}{\NS^b}, \; \bLambda \right ) \cong \faktor{\bLambda}{\Fitt^0_{\bLambda}(H^2_{\Sigma,\Iw}(\cO_{L,S}, T))^{\ast \ast}}.
            \end{align*}
        }
        \item{There exists a perfect pairing of $\bLambda$-modules
            \begin{align*}
                \faktor{\NS^{r_T}}{\NS^b} \times \faktor{\bLambda}{\Fitt^0_{\bLambda}(H^2_{\Sigma,\Iw}(\cO_{L,S}, T))^{\ast \ast}} \to \faktor{Q(\bLambda)}{\bLambda}
            \end{align*}
            which is explicitly given by the assignment $(u,v) \mapsto \overline{v}\cdot\eta^*(\overline{u})$, where $\eta$ is any choice of $\bLambda$-basis of $\NS^b$ and $\overline{u}$ and $\overline{v}$ are any lifts of $u$ and $v$ to $\NS^{r_T}$ and $\bLambda$, respectively.
        }
        \item{There exists an injective pseudo-isomorphism of $\Lambda$-modules
            \begin{align*}
                \faktor{\NS^{r_T}}{\NS^b} \approx \left ( \faktor{\bLambda}{\Fitt^0_{\bLambda}(H^2_{\Sigma,\Iw}(\cO_{L,S}, T))} \right)^\circ.
            \end{align*}
            If $p \nmid [L:K]$ then this can be taken to be a pseudo-isomorphism of $\bLambda$-modules where $\circ$ now also inverts the $\cG$-action.
        }
    \end{liste}
\end{thm}

\begin{rk}\text{}\begin{enumerate}[label=(\alph*)]
    \item{The existence of pairings of the displayed shape in part (2) of Theorem \ref{pairing-theorem} was first observed (at least in the case of representations with coefficients in Gorenstein orders in finite-dimensional $\cQ$-algebras) by Burns, Sano and Tsoi in \cite[]{bst}. The above result can therefore be seen as a natural Iwasawa-theoretic analogue of the pairing constructed for $T$ by the aforementioned authors.
    Similar Iwasawa-theoretic results in the setting of $K =\Q$ have previously appeared in \cite[Prop.\@ 3]{NQD14}, which in turn is based on \cite{KraftSchoof}, and \cite[Thm.\@ 4]{Solomon2014}.
    }
    \item{Let $M$ be a finitely generated $\bLambda$-module. The ideal $\Fitt_{\bLambda}^0(M)^{\ast \ast}$ is related to a generalised notion of characteristic ideal due to Sakamoto in \cite[Appendix C]{Sakamoto20} (see also the earlier construction of Greither in \cite[\S\,5.2]{Cornelius}). To briefly explain this, fix a presentation $0 \to N \xrightarrow{f} \bLambda^n \to M \to 0$ of $M$. Define $\mathrm{char}_\bLambda(M)$ to be the image of $\bidual_{\bLambda}^n N$ inside $\bidual_{\bLambda}^n \bLambda^n \cong \bLambda$ under the map induced by $f$. Then $\mathrm{char}_{\bLambda}(M)$ is independent of the choice of presentation of $M$, reflexive and one has an inclusion $\Fitt_\bLambda^0(M)^{\ast \ast} \subseteq \mathrm{char}_\bLambda(M)$. Moreover, we have an equality $\Fitt_{\bLambda}^0(M)^{\ast \ast} = \mathrm{char}_\bLambda(M)$ in the following scenarios:
        \begin{enumerate}[label=(\roman*)]
            \item{$M$ is cyclic. In this case one can take $N = \Fitt_\bLambda(M)$ and $n = 1$ in the definition of $\mathrm{char}_\bLambda(M)$.}
            \item{The projective dimension of $M$ is at most one. In this case $N$ is projective and the natural map $\exprod_\bLambda^r N \to \bidual_{\bLambda}^r N$ is an isomorphism.}
            \item{$p \nmid |\cG|$. In this case $\bLambda$ is a normal ring and $\mathrm{char}_{\bLambda}(M)$ coincides with the usual notion of characteristic ideal.} 
            \item{The $\mu$-invariant of $M$ (as a $\Lambda$-module) vanishes. To see this, note that by Lemma \ref{IwasawaInvariants}\,(a) we have $M_\fp = 0$ for any singular prime $\fp$ of $\bLambda$. On the other hand, if $\fp$ is a regular prime then the ring $\bLambda_\fp$ is normal and so $\Fitt_{\bLambda_{\fp}}^0(M_\fp)^{\ast \ast} = \mathrm{char}_{\bLambda_{\fp}}(M_\fp)$. Since $\mathrm{char}_\bLambda(M)$ and $\Fitt_{\bLambda}^0(M)^{\ast \ast}$ are both reflexive and any reflexive $\bLambda$-module is determined by its localisations at primes of height at most one (see \cite[Lem.\@ C.13]{Sakamoto20}), the claim follows.}
        \end{enumerate}
    }
    \end{enumerate}
\end{rk}

\textit{Proof of Theorem \ref{pairing-theorem}:}
    By applying Lemma \ref{LittleLemma}\,(d) to the exact sequence (\ref{yoneda-extension-sequence-Iwasawa-modified}) we see that $\NS^b = \eta \bLambda$ is a free $\bLambda$-module of rank one and that we have an isomorphism
    \begin{align} \label{isom}
    (\NS^{r_T})^\ast \stackrel{\simeq}{\longrightarrow} I^{\ast \ast},
    \quad f \mapsto f (\eta). 
    \end{align}
    where we denote $I := \Fitt^0_{\bLambda} ( H^2_{\Sigma, \Iw} (\bigO_{L, S}, T))$. Moreover, (\ref{Exts}) and Theorem \ref{UN-structure-theorem}(a) taken together imply that the module $\Ext_{\bLambda}^1(\NS^{r_T}, \bLambda)$ vanishes. We thus have a commutative diagram with exact rows
    \begin{center}
        \begin{tikzcd}
            0 \arrow[r] &\displaystyle(\NS^{r_T})^* \arrow[r] \arrow{d}{\simeq} &
            ( \bLambda \eta )^* \arrow[r] \arrow{d}{\simeq} &\Ext^1_\bLambda \left  (\faktor{\NS^{r_T}}{\NS^b}, \bLambda \right) \arrow[r] \arrow[d] & 0\\
            0 \arrow[r] &I^{\ast \ast} \arrow[r] &\bLambda \arrow[r] &\faktor{\bLambda}{I^{\ast \ast}} \arrow[r] &0
        \end{tikzcd}
    \end{center}
    The isomorphism given in (a) now follows via an application of the Five-Lemma. \medskip \\
    To prove (c), note that $\NS^{r_T}/\NS^b$ has projective dimension one as a $\Lambda$-module and thus has no non-zero finite submodules (see, for example, \cite[Prop.\@ 5.3.19]{NSW}). We therefore have an injective pseudo-isomorphism
    \begin{align}\label{pseudo-iso-1}
         \faktor{\NS^{r_T}}{\NS^b} \to \left ( \faktor{\bLambda}{I^{\ast \ast}} \right)^\circ
    \end{align}
    which can be obtained, for example, by appealing to \cite[Prop.\@ 5.5.13]{NSW} or via an explicit calculation using a free resolution of $\NS^{r_T}/\NS^b$ (in which case one in fact obtains an isomorphism). Now, the natural map $\bLambda/I \to \bLambda/I^{\ast \ast}$ is a pseudo-isomorphism of $\Lambda$ modules as its kernel is $I^{\ast \ast}/I$ which is pseudo-null since $I^{\ast \ast}$ is the reflexive hull of $I$. On the other hand, $I^{\ast \ast}$ is $\Lambda$-free of the same rank as $\bLambda$ by (\ref{isom}) whence $\bLambda/I$ is $\Lambda$-torsion and so one may deduce the existence of a pseudo-isomorphism $\bLambda/I^{\ast \ast} \approx \bLambda/I$.
    The pseudo-isomorphism of the Theorem now follows by inverting the $\Gamma$-action across this map and then composing with the map (\ref{pseudo-iso-1}).\medskip \\
   It remains to demonstrate the existence of the pairing in (b). Observe that if $M$ is $\bLambda$-torsion, then by applying the functor $\Hom_\bLambda(M,-)$ to the tautological sequence $0 \to \bLambda \to Q(\bLambda) \to Q(\bLambda)/\bLambda \to 0$ one finds a canonical identification $M^\lor \cong \Ext_\bLambda^1(M, \bLambda)$. In addition, one knows by 
   \cite[Prop.\@ 5.5.8\,(iv)]{NSW} and the isomorphism (\ref{Exts})
   that
    \begin{align} \label{double-ext}
        \left (\faktor{\bLambda}{I^{\ast \ast}}\right )^\lor \cong \left ( \faktor{\NS^{r_T}}{\NS^b}\right)^{\lor\lor} \cong \faktor{\NS^{r_T}}{\NS^b}.
    \end{align}
    These two facts taken together establish both the existence and the perfectness of the desired pairing. Since the quotient $\NS^{r_T}/\NS^b$ is $\bLambda$-torsion, we may regard $\eta^*$ as being an element of $(\NS^{r_T})^* \otimes_\bLambda Q(\bLambda)$. A straightforward calculation then shows that one can in fact regard this as an element of $(\NS^{r_T}/\NS^b)^\lor$ from which one deduces the given explicit description of the pairing.
\qed

\subsection{Results on finite level}\label{finite-level-section}

In analogy to the Iwasawa-theoretic definition of basic norm coherent sequences, it is natural to make the following corresponding definition on finite level.

\begin{definition}
For each $n \in \NN_0$, we define the $\cR[\cG_n]$-module of \emph{basic universal norms} $\UN^b_n = \UN^b_n(T,L_\infty)$ to be the image of $\NS^b(T,L_\infty)$ under the map of Theorem \ref{UN-structure-theorem}\,(b).
\end{definition}

There is another, equivalent, way of constructing the module of basic universal norms that is closer in spirit to the definition of basic norm-coherent sequences. In order to explain this, we let $e_{L_n, T} \in \mathcal{Q}[\cG_n]$ be the sum of the primitive idempotents that annihilate $H^2_\Sigma (\bigO_{L_n, S}, T)$. In this regard we remark that Jannsen has conjectured in \cite[Conj.~1]{jannsen} that this module should be finite in all but a few exceptional cases. We then have the projection map
\begin{align*}
    \Theta_{L_n} \: \Det_{\cR [\cG_n]} (C^\bullet_n) & \hookrightarrow
    \mathcal{Q} \otimes_{\cR} \Det_{\cR [\cG_n]} (C^\bullet_n) \\
    & \stackrel{\simeq}{\longrightarrow} \Det_{\mathcal{Q}[\cG_n]} ( \mathcal{Q} \otimes_{\cR} C^\bullet_n) \\
    & \stackrel{\simeq}{\longrightarrow} \Det_{\mathcal{Q} [\cG_n]} ( \mathcal{Q} \otimes_\cR H^0 (C^\bullet_n)) \otimes_{\cQ [\cG_n]} ( \Det_{\cQ [\cG_n]} ( \cQ \otimes_{\cR} H^1 (C^\bullet_n)))^{-1} \\
    & \stackrel{\cdot e_{L_n, T}}{\longrightarrow} 
    e_{L_n, T} \Big ( \Det_{\mathcal{Q} [\cG_n]} ( \mathcal{Q} \otimes_\cR H^0 (C^\bullet_n)) \otimes_{\cQ [\cG_n]}  ( \Det_{\cQ [\cG_n]} ( \cQ \otimes_{\cR} H^1 (C^\bullet_n)))^{-1} \Big ) \\
    & \stackrel{\simeq}{\longrightarrow} e_{L_n, T}  \Big ( \exprod^{r_T}_{\cQ [\cG_n]} \cQ \otimes_\cR H^1_\Sigma (\bigO_{L, S}, T) \Big )
    \otimes_{\cQ [\cG_n]}  \exprod^{r_T}_{\cQ [\cG_n]} ( Y_K (T) \otimes_\cR \cQ [\cG_n])^\ast \\
    & \stackrel{\simeq}{\longrightarrow} 
    e_{L_n, T}  \Big ( \exprod^{r_T}_{\cQ [\cG_n]} \cQ \otimes_\cR H^1_\Sigma (\bigO_{L, S}, T) \Big ),
\end{align*}
where the second arrow follows from the base-change property of determinant functors, the third from the natural passage-to-cohomology map, the fourth by multiplication by the idempotent $e_{L_n, T}$, and the final one by applying the (non-canonical) isomorphism $ \exprod^{r_T}_{\cQ [\cG_n]} ( Y_K (T) \otimes_\cR \cQ [\cG_n])^\ast \cong \cQ [\cG_n]$ resulting from the fact that $Y_K (T) \otimes_\cR \cR [\cG_n]$ is a free $\cR [G_n]$-module of rank $r_T$. 

\begin{lem} \label{FiniteLem}
The image of the map $\Theta_{L_n}$ is contained in $\left( \bidual^{r_T}_{\cR [\cG_n]} H^1_\Sigma (\bigO_{L_n, S}, T) \right ) [1 - e_{L_n, T}]$, and coincides with $\UN^b_n$.
\end{lem}

\begin{proofbox}
Since the complex $C^\bullet_{n}$ admits a standard quadratic representative with respect to the map $H^1(C^\bullet_{n}) \to Y_K(T)^* \otimes_\cR \cR[\cG_n]$, the first claim is exactly \cite[Prop.\@ A.7\,(i)]{EulerSystemsSagaI}.
As for the second claim, we observe that the descriptions of $\Theta$ and $\Theta_{L_n}$ in terms of rank reduction maps $\pi_{\psi_\infty}$ and $\pi_{\psi_n}$, respectively, yield 
a commutative diagram
\begin{cdiagram}
\Det_\bLambda ( C^\bullet_\infty) \arrow{r}{\Theta} \arrow[twoheadrightarrow]{d} & 
\bidual^{r_T}_\bLambda H^1_{\Sigma, \Iw} (\bigO_{L, S}, T) \arrow{d} \\
\Det_{\cR [\cG_n]} ( C^\bullet_n) \arrow{r}{\Theta_{L_n}} &  
\bidual^{r_T}_{\cR [\cG_n]} H^1_\Sigma (\bigO_{L_n, S}, T),
\end{cdiagram}
where the vertical arrows are the natural codescent maps.
\end{proofbox}

It turns out that the module of basic universal norms may not be very interesting at the bottom layers of the extension $L_\infty$. We remark however that the behaviour explicated in the next Lemma will always stop if one climbs high enough up the tower due to our assumption that no finite place splits completely in $L_\infty | K$. 

\begin{lem}
Suppose there is a finite place $v \in S$ that splits completely in $L_n | K$ and is such that $H^0 (K_v, T^\ast (1))$ is non-zero. Then $\UN^b_n = 0$.
\end{lem}

\begin{proofbox}
We have a commutative diagram
\begin{cdiagram}
\NS^{r_T} \arrow[hookrightarrow]{r} \arrow{d} & \exprod^{r_T}_\bLambda \Pi \arrow{d} & \\
\UN^{r_T}_n \arrow[hookrightarrow]{r} & \exprod^{r_T}_{\cR [\cG_n]} \Pi_n,
\end{cdiagram}
where the vertical maps are the natural codescent maps. It therefore suffices to demonstrate that any basis $\eta$ of $\NS^b$ is contained in $I_{\Gamma^n} \cdot \exprod^{r_T}_{\bLambda} \Pi$, where we write
\[
I_{\Gamma^n} = \ker \{ \bLambda \to \cR [\Gamma_n] \} 
\]
for the augmentation ideal relative to $\Gamma^n$. An application of \cite[Prop.\@ A.2\,(ii)]{EulerSystemsSagaI} to the exact sequence (\ref{yoneda-extension-sequence-Iwasawa-modified}) implies that for any $f \in \exprod^{r_T}_\bLambda \Pi^\ast$ we have 
\[
f (\eta) \in \Fitt^0_\bLambda ( H^2_{\Sigma, \Iw} ( \bigO_{L, S}, T)). 
\]
If $b_1, \dots, b_d \in \Pi$ constitutes a $\bLambda$-basis, then for any $\sigma \in \mathfrak{S}_{d, r_T}$ (see the definiton following (\ref{ExplicitFormula})) the inclusion
\[
(b_{\sigma (1)}^\ast \wedge \dots \wedge b_{\sigma({r_T})}^\ast ) (\eta) \in \Fitt^0_\bLambda ( H^2_{\Sigma, \Iw} ( \bigO_{L, S}, T)) 
\]
holds. 
Since $\{ \bigwedge_{1 \leq i \leq r_T} b_{\sigma (i)} \mid \sigma \in \mathfrak{S}_{d, r_T} \}$ constitutes a basis of $\exprod^{r_T}_\bLambda \Pi$, we conclude that
\[
\eta \in \Fitt^0_\bLambda ( H^2_{\Sigma, \Iw} ( \bigO_{L, S}, T)) \cdot \exprod^{r_T}_\bLambda \Pi.
\]
We are therefore reduced to showing that $\Fitt^0_\bLambda ( H^2_{\Sigma, \Iw} ( \bigO_{L, S}, T)) \subseteq I_{\Gamma^n}$. By Proposition \ref{SplitPrimesProp} there is a surjection
\[
H^2_{\Sigma, \Iw} (\bigO_{L, S}, T) \twoheadrightarrow H^2_\Sigma (\bigO_{L_n, S}, T) \twoheadrightarrow H^0 (K_v, T^\ast (1)) \otimes_\cR \cR [\cG_n].
\]
Since we assumed $H^0 (K_v, T^\ast (1)) \neq 0$, the module $H^0 (K_v, T^\ast (1)) \otimes_\cR \cR [\cG_n]$ is $\cR [\cG_n]$-free of non-zero rank $t$, say. The Lemma therefore follows from the inclusion
\[
\Fitt^0_\bLambda ( H^2_{\Sigma, \Iw} ( \bigO_{L, S}, T)) \subseteq
\Fitt_\bLambda^0 ( H^0 (K_v, T^\ast (1)) \otimes_\cR \cR [\cG_n]) = \Fitt^0_\bLambda ( \cR [\cG_n]^{t} ) = I_{\Gamma^n}^t. \tag*{\qedhere} 
\]
\end{proofbox}

We next turn to a description of the quotient $\bidual^{r_T}_{\cR [\cG_n]} H^1_\Sigma (\bigO_{L, S}, T) / \UN^{r_T}_n$. This should be regarded as a complement to the study of the quotient $\bidual^{r_T}_{\cR [\cG_n]} H^1_\Sigma (\bigO_{L, S}, T) / \UN^{b}_n$ undertaken in \cite{bst}. We recall that for an abelian group $A$ we write $A_\tor$ and $A_\tf$ for its torsion and torsion-free parts, respectively.

\begin{proposition} \label{Dodgy-Proposition}
Assume that $p \nmid |\cG|$. 
\begin{liste}
\item We have an equality
\[
\Big\{ f (a) \mid a \in \UN^{r_T}_n, \; f \in \exprod^{r_T}_{\cR [\cG_n]} H^1_\Sigma ( \bigO_{L_n, S}, T)^\ast  \Big\} = \Fitt^0_{\cR [\cG_n]} \big ( H^2_{\Sigma, \Iw}( \bigO_{L, S}, T)^{\Gamma^n, \vee}_\tor \big ).
\]
\item The following holds:
\[
\left ( \faktor{\bidual^{r_T}_{\cR [\cG_n]} H^1_\Sigma ( \bigO_{L_n, S}, T)}{\UN^{r_T}_n} \right )_{\tor} \cong \left ( \faktor{\cR [\cG_n]}{\Fitt^0_{\cR [\cG_n]} ( H^2_{\Sigma, \Iw}( \bigO_{L, S}, T)^{\Gamma^n, \vee}_\tor )} \right)^\vee.
\]
\item There is an injection
\[
\left ( \faktor{\exprod^{r_T}_{\cR [\cG]} H^1_\Sigma ( \bigO_{L, S}, T)}{\UN^{r_T}_0} \right )_{\tf} \hookrightarrow ( H^2_{\Sigma, \Iw} (\bigO_{L, S}, T)^\Gamma)_\tf \otimes_{\cR [\cG]} \exprod^{r_T - 1}_{\cR [\cG]} H^1_\Sigma ( \bigO_{L, S}, T)
\]
that is induced by the boundary morphism
\[
\delta \: \left ( \faktor{
H^1_\Sigma (\bigO_{L, S}, T) 
}{\UN^1_0  } \right )_\tf
\stackrel{\simeq}{\longrightarrow} H^2_{\Sigma, \Iw} (\bigO_{L, S}, T)^\Gamma_\tf.
\]
\end{liste}
\end{proposition}

\begin{proofbox}
By truncating the exact sequence representing the complex $C^\bullet_\infty$ we obtain an exact sequence
\begin{cdiagram}
   0 \arrow{r} & H^1_{\Sigma, \Iw} ( \bigO_{L, S}, T) \arrow{r} & \Pi \arrow{r}{\psi} & \im \psi \arrow{r} & 0 .
\end{cdiagram}
Since $\im(\psi)$ has trivial $\Gamma^n$-invariants, Theorem \ref{UN-structure-theorem}\,(b) implies that this sequence descends to give an exact sequence
\begin{equation} \label{DescendedSequence}
\begin{tikzcd}
   0 \arrow{r} & \UN^1_n \arrow{r} & \Pi_n \arrow{r} & (\im \psi)_{\Gamma^n} \arrow{r} & 0.
\end{tikzcd}
\end{equation}
Dualising, and using the identification $\Ext^1_{\cR [\cG_n]} ( (\im \psi)_{\Gamma_n}, \cR [\cG_n]) \cong ( (\im \psi)_{\Gamma_n,\tor})^\vee$, we get the exact sequence
\begin{equation} \label{cor-sequence-1}
\begin{tikzcd}
\Pi_n^\ast \arrow{r} & (\UN^1_n)^\ast \arrow{r} & ((\im \psi)_{\Gamma_n, \tor} )^\vee \arrow{r} & 0.
\end{tikzcd}
\end{equation}
We have observed in Remark \ref{structure-theorem-remark} (b) that the assumption $p \nmid |\cG |$ implies that $\UN^1_n$ is $\cR [\cG_n]$-free of rank $r_T$, so the above exact sequence is in fact a free presentation of $((\im \psi)_{\Gamma^n, \tor})^\vee$. This implies that
\begin{align*}
\im \big \{ \exprod^{r_T}_{\cR [\cG_n]} 
\Pi_n^\ast \to \exprod^{r_T}_{\cR [\cG_n]} (\UN^1_n )^\ast \cong \cR [\cG_n] \big \} 
& = \Fitt^0_{\cR [\cG_n]} \left ( (\im \psi)_{\Gamma^n, _\tor}^\vee \right). 
\end{align*}
Combining this with the identification $\UN^{r_T}_n = \bidual^{r_T}_{\cR [\cG_n]} \UN^1_n$ from Theorem \ref{UN-structure-theorem}\,(c) then gives 
\[
\Big \{ f (a) \mid a \in \UN^{r_T}_n, \; f \in \exprod^{r_T}_{\cR [\cG_n]} \Pi_n^\ast \Big\} = \Fitt^{0}_{\cR [\cG_n]} ( (\im \psi)_{\Gamma^n, \tor}^\vee). 
\]
The exact sequence
\begin{cdiagram}
   0 \arrow{r} & \im \psi \arrow{r} & \Pi \arrow{r} & H^1 (C^\bullet_\infty ) \arrow{r} & 0
\end{cdiagram}
gives the exact sequence
\begin{cdiagram}
0 \arrow{r} & H^1 (C^\bullet_\infty)^{\Gamma^n} \arrow{r} & (\im \psi )_{\Gamma^n} \arrow{r} & \Pi_n.
\end{cdiagram}
Since $\Pi_n$ is torsion-free, we deduce that 
\begin{equation} \label{cor-sequence-2}
(\im \psi)_{\Gamma^n, \tor} = H^1 (C^\bullet_\infty)^{\Gamma^n}_\tor = H^2_{\Sigma, \Iw}( \bigO_{L, S}, T)^{\Gamma^n}_\tor. 
\end{equation}

Finally, since the cokernel of $H^1_\Sigma (\bigO_{L_n, S}, T) \hookrightarrow \Pi_n$ is $\cR$-torsion free, the restriction map
$
\Pi_n^\ast \to H^1_\Sigma (\bigO_{L_n, S}, T)^\ast
$
is surjective. This shows that 
\[
\Big \{ f (a) \mid a \in \UN^{r_T}_n, \; f \in \exprod^{r_T}_{\cR [\cG_n]} \Pi_n^\ast \Big\}
= 
\Big\{ f (a) \mid a \in \UN^{r_T}_n, \; f \in \exprod^{r_T}_{\cR [\cG_n]} H^1_\Sigma (\bigO_{L_n, S}, T)^\ast \Big\}
\]
and concludes the proof of (a). For (b), let $C$ denote the cokernel of the map $\UN^{r_T}_n \to \bidual^{r_T}_{\cR [\cG_n]} H^1_\Sigma (\bigO_{L_n, S}, T)$. Then dualising gives a commutative diagram
\begin{cdiagram}[column sep=tiny]
   & \left ( \bidual^{r_T}_{\cR [\cG_n]} H^1_\Sigma (\bigO_{L_n, S}, T) \right)^\ast \arrow{r} \arrow[twoheadrightarrow]{d} & (\UN^{r_T}_n)^\ast \arrow{r} \arrow{d}{\simeq} & \Ext^1_{\cR [\cG_n]} ( C, \cR [\cG_n] ) \arrow{r} \arrow{d} & 0 \\
   0 \arrow{r} & \Fitt^0_{\cR [\cG_n]} ( H^1 (C^\bullet_\infty)^{\Gamma^n, \vee}_\tor ) \arrow{r} &
   \cR [\cG_n] \arrow{r} & \faktor{\cR [\cG_n]}{\Fitt^0_{\cR [\cG_n]} ( H^1 (C^\bullet_\infty)^{\Gamma^n, \vee}_\tor )} \arrow{r} & 0,
\end{cdiagram}
where the middle isomorphism is given by evaluating at a generator of the free rank one module $\UN^{r_T}_n$ and surjectivity of the first of these follows from the surjectivity of the natural map
\[
\exprod^{r_T}_{\cR [\cG_n]} H^1_\Sigma (\bigO_{L_n, S}, T)^\ast \to 
\left ( \bidual^{r_T}_{\cR [\cG_n]} H^1_\Sigma (\bigO_{L_n, S}, T) \right)^\ast, \quad
f \mapsto \{ a \mapsto f (a) \},
\]
itself a consequence of the proof of (\ref{surjectivity-display}). 
Applying the snake lemma to the aforementioned diagram reveals that the rightmost downward map is an isomorphism as well. The claim follows now upon noting that
\[
\Ext^1_{\cR [\cG_n]} ( C, \cR [\cG_n] ) \cong \Ext^1_{\cR} ( C, \cR ) \cong ( C_\tor)^\vee. 
\]
Turning our sights now to (c), we first record that the spectral sequence (\ref{spectral-sequence}) applied to the complex $C^\bullet_\infty$ gives an exact sequence
\begin{cdiagram}
   0 \arrow{r} & \UN^1_0 \arrow{r}{\iota} & H^1_\Sigma (\bigO_{L, S}, T) \arrow{r}{\delta} & H^2_{\Sigma,\Iw} (\bigO_{L, S}, T)^\Gamma \arrow{r} & 0.
\end{cdiagram}
Dualising this sequence, we obtain 
\begin{cdiagram}
  0 \arrow{r} & ( H^2_{\Sigma,\Iw} (\bigO_{L, S}, T)^\Gamma)^\ast \arrow{r} &  H^1_\Sigma (\bigO_{L, S}, T)^\ast \arrow{r} & \im \iota^\ast \arrow{r} & 0, 
\end{cdiagram}
where $\iota^\ast$ denotes the dual map of $\iota$. This induces the exact sequence
\begin{cdiagram}[column sep=tiny]
   ( H^2_{\Sigma, \Iw} (\bigO_{L, S}, T)^\Gamma)^\ast \otimes_{\cR [\cG]} \exprod^{r_T - 1}_{\cR [\cG]} H^1_\Sigma (\bigO_{L, S}, T)^\ast \arrow{r} 
   &  
   \exprod^{r_T}_{\cR [\cG]} H^1_\Sigma (\bigO_{L, S}, T)^\ast 
   \arrow{r}
   &\exprod^{r_T}_{\cR [\cG]} \im \iota^\ast \arrow{r} & 0 .
\end{cdiagram}
Dualising again, we find that there is an exact sequence
\begin{cdiagram}[column sep=tiny]
   0 \arrow{r} & 
   \Big ( \exprod^{r_T}_{\cR [\cG]} \im \iota^\ast \Big)^\ast \arrow{r} & \exprod^{r_T}_{\cR [\cG]} H^1_\Sigma (\bigO_{L, S}, T)
   \arrow{r}
   & \left ( ( H^2_{\Sigma, \Iw} (\bigO_{L, S}, T)^\Gamma)^\ast \otimes_{\cR [\cG]} \exprod^{r_T - 1}_{\cR [\cG]} H^1_\Sigma (\bigO_{L, S}, T)^\ast \right)^\ast . 
\end{cdiagram}
Since $p \nmid |\cG|$, the module $H^1_\Sigma (\bigO_{L, S}, T)^*$ is $\cR [\cG]$-projective, so we have 
\begin{align*}
 & \phantom{=} \left ( ( H^2_{\Sigma, \Iw} (\bigO_{L, S}, T)^\Gamma)^\ast \otimes_{\cR [\cG]} \exprod^{r_T - 1}_{\cR [\cG]} H^1_\Sigma (\bigO_{L, S}, T)^\ast \right)^\ast \\
& = 
( H^2_{\Sigma, \Iw} (\bigO_{L, S}, T)^\Gamma)^{\ast \ast} \otimes_{\cR [\cG]} \left ( \exprod^{r_T - 1}_{\cR [\cG]} H^1_\Sigma (\bigO_{L, S}, T)^\ast \right)^\ast \\
& \cong ( H^2_{\Sigma, \Iw} (\bigO_{L, S}, T)^\Gamma)_\tf 
\otimes_{\cR [\cG]} \exprod^{r_T - 1}_{\cR [\cG]} H^1_\Sigma (\bigO_{L, S}, T).
\end{align*}
It therefore remains to show that $\faktor{\exprod^{r_T}_{\cR [\cG]} H^1_\Sigma (\bigO_{L, S}, T)}{\big ( \exprod^{r_T}_{\cR [\cG]} \im \iota^\ast \big)^\ast}$ is exactly the torsion-free part of $\faktor{\exprod^{r_T}_{\cR [\cG]} H^1_\Sigma (\bigO_{L, S}, T)}{\UN^{r_T}_0}$. \\

The above exact sequence shows that the former quotient is torsion-free. From the exact sequence 
\begin{cdiagram}
   0 \arrow{r} & \im \iota^\ast \arrow{r} & (\UN^1_0)^\ast \arrow{r} & 
   \Ext^1_{\cR [\cG]} ( H^2_{\Sigma, \Iw} (\bigO_{L, S}, T)^\Gamma, \cR [\cG]) \arrow{r} & 0
\end{cdiagram}
we see that $\im \iota^\ast$ has finite index inside $(\UN^1_0)^\ast$. Finally, from the diagram
\begin{cdiagram}
   \mathcal{Q} \otimes_\cR \exprod^{r_T}_{\cR [\cG]} \im \iota^\ast \arrow{r}{\simeq} & \mathcal{Q} \otimes_\cR \exprod^{r_T}_{\cR [\cG]}  (\UN^1_0)^\ast \\
   \exprod^{r_T}_{\cR [\cG]} \im \iota^\ast \arrow{u} \arrow{r} & \exprod^{r_T}_{\cR [\cG]} (\UN^1_0)^\ast \arrow{u}
\end{cdiagram}
we get, via dualising, the commutative diagram
\begin{cdiagram}
   \mathcal{Q} \otimes_\cR \Big ( \exprod^{r_T}_{\cR [\cG]} \im \iota^\ast \Big )^\ast  & \mathcal{Q} \otimes_\cR \exprod^{r_T}_{\cR [\cG]}  \UN^1_0 \arrow{l}{\simeq} \\
   \Big ( \exprod^{r_T}_{\cR [\cG]} \im \iota^\ast \Big)^\ast \arrow[hookrightarrow]{u}  & \exprod^{r_T}_{\cR [\cG]} \UN^1_0 \arrow[hookrightarrow]{u} \arrow{l}
\end{cdiagram}
from which we deduce that $\UN^{r_T}_0 = \exprod^{r_T}_{\cR [\cG]} \UN^1_0$  injects with finite index into $\big ( \exprod^{r_T}_{\cR [\cG]} \im \iota^\ast \big)^\ast$. 
\end{proofbox}

\begin{rk}
A curious consequence of the proof of Proposition \ref{Dodgy-Proposition} is that the $\Z_p$-torsion submodule of $H^2_{\Sigma, \Iw} (\bigO_{L, S}, T)^{\Gamma}$ is necessarily a cyclic $\cR [\cG]$-module if $r_T = 1$ and $p \nmid | \cG|$. Indeed, in this situation $\UN^1_0$ is $\cR [\cG]$-free of rank one, so (\ref{cor-sequence-1}) and (\ref{cor-sequence-2}) imply that  $H^2_{\Sigma, \Iw} (\bigO_{L, S}, T)^{\Gamma, \vee}_\tor$ is cyclic. 
The ideal $\Fitt^0_{\cR [\cG]} ( H^2_{\Sigma, \Iw}( \bigO_{L, S}, T)^{\Gamma, \vee}_\tor )$ is principal and generated by a non-zero divisor $x$, say (see \cite[Prop.\@ 2.2.2]{Cornelius}), hence it is then immediate from the exact sequence
\begin{cdiagram}
0 \arrow{r} & \cR [\cG] \arrow[r, "\cdot x"] & \cR [\cG] \arrow{r} & \faktor{\cR [\cG]}{(x)} \arrow{r} & 0
\end{cdiagram}
that $\Ext^1_{\cR [\cG]} (\faktor{\cR [\cG]}{(x)}, \cR [\cG]) = H^2_{\Sigma, \Iw} (\bigO_{L, S}, T)^{\Gamma}_\tor$ is $\cR [\cG]$-cyclic. 
\end{rk}

\markboth{Applications to arithmetic}{Tate twist $T=\Z_p (1)$}
\section{Applications to arithmetic}

In this section we exemplify how one can use the general framework laid out in the previous section to derive concrete arithmetic consequences. 

\subsection{Tate twist $T = \Z_p (1)$}

We shall first specialise to the representation $T = \Z_p (1)$. After an appropriate choice of set $\Sigma$, this representation
satisfies Hypothesis \ref{main-hypothesis} as observed in Example \ref{RepExamples}\,(b) and so the general results from the last section are applicable in this situation. 
We introduce the following additional notation. \medskip \\
Let $F$ be a number field. For finite sets  of places $V$ and $\Sigma$ of $F$ satisfying $V \cap \Sigma = \emptyset$, $S_\infty (K) \cap \Sigma = \emptyset$, and $S_p(F) \subseteq V$ we denote
\begin{itemize}
    \item $U_{F, V, \Sigma} = \Z_p \otimes_\Z \bigO_{F, V, \Sigma}^\times \cong H^1_\Sigma ( \bigO_{F, S}, \Z_p (1))$ the $p$-completion of the $(V, \Sigma)$-unit group of $F$, 
    \item $A_{V,\Sigma} (F) = \Z_p \otimes_\Z \cl_{V,\Sigma} (F)$ the $p$-Sylow subgroup of the $V$-ray class group mod $\Sigma$ of $F$,
    \item $Y_{F, V} = \bigoplus_{v \in V} \Z_p$ the free $\Z_p$-module on the set of places contained in $V$,
    \item $X_{F, V}$ the kernel of the natural augmention map $Y_{F, V} \to \Z_p$.
\end{itemize}
When $\Sigma = \emptyset$, we omit the reference to $\Sigma$ from the above notation. If $F_\infty = \bigcup_{n \geq 0} F_n$ defines a $\Z_p$-extension of $F$, then for any of the objects $\square$ above we denote by $\square^\infty$ the projective limit over $n$ of the respective objects $\square$ for $F_n$, where the limits are taken with respect to the natural transition maps in each situation.

\subsubsection{Iwasawa Main Conjecture}

In this section we give a straightforward application of the result of Theorem \ref{pairing-theorem}. In particular, we show that the isomorphism of Theorem \ref{pairing-theorem}\,(a) refines the (plus-part of the) classical Iwasawa Main Conjecture.\medskip \\
For any integer $m \geq 1$, let $\xi_m = e^{2 \pi i / m}$, which we regard as an element of $\overline{\Q}$ via a fixed embedding $\overline{\Q} \hookrightarrow \C$. 
Take $f >0$ to be an integer such that $f \not \equiv 2 \mod 4$ and
$p \nmid f$. Then for every $n \geq 0$ we set $L_n = \Q (\xi_{f p^{n + 1}})$ and note that the collection of maximal totally real subfields $L_n^+$ constitutes a $\Z_p$-extension that satisfies the assumptions of \S\,\ref{set-up-section}. We will therefore resume the notation introduced there and hope that this does not cause any confusion. \medskip \\
For every $n$ we denote by $\Cyc_n$ the group of cyclotomic units of $L_n$. In other words, 
\begin{align*}
    \Cyc_n := \langle -1, 1 - \xi_d : d \mid f p^{n+1} \rangle_{\Z_p [\cG_n]} \cap \bigO_{L_n}^\times. 
\end{align*}
We then set $\Cyc^\infty := \varprojlim_n \Cyc_n$, where the limit is taken with respect to the norm maps.

\begin{thm} \label{IMC-refinement}
    Let $\chi \in \widehat{\cG}$ be an even character and assume that $p \nmid [L : \Q]$, where $L = L_0$. 
    Then there is an isomorphism of $\Lambda_\chi := \ZZ_p(\im(\chi))\llbracket \Gamma \rrbracket$-modules
    \begin{align*}
        \faktor{U^{\infty,\chi}}{\Cyc^{\infty,\chi}} \cong \alpha  \left ( \faktor{\Lambda_\chi}{\cchar_{\Lambda_\chi}(A^{\infty,\chi})} \right),
    \end{align*}
    where $(-)^\chi$ is the functor $- \otimes_{\ZZ_p[\cG]} \ZZ_p(\im(\chi)) = - \otimes_\bLambda \Lambda_\chi$, and $\alpha (-) = \Ext^1_{\Lambda_\chi} ( -, \Lambda_\chi)$ denotes the Iwasawa adjoint. 
\end{thm}

\begin{rk} \label{IMC-refinement-remark}
Theorem \ref{IMC-refinement} can be seen as a refinement of the classical Iwasawa Main Conjecture for the following reason: \cite[Prop.\@ 5.5.13]{NSW} gives a pseudo-isomorphism $\alpha  (\Lambda_\chi/\text{char}_{\Lambda_\chi}(A^{\infty,\chi}))
\approx (\Lambda_\chi/\cchar_{\Lambda_\chi}(A^{\infty, \chi}))^\circ
$ and so taking characteristic ideals on both sides of the isomorphism stated in Theorem \ref{IMC-refinement} yields
\[
\mathrm{char}_{\Lambda_\chi}\left ( \faktor{U^{\infty, \chi}}{\Cyc^{\infty, \chi}} \right) = \mathrm{char}_{\Lambda_\chi}(A^{\infty,\chi}).
\]
This is one form of the classical Iwasawa Main Conjecture first proved by Mazur and Wiles \cite{MazurWiles} (see also \cite[Thm.~5.1]{LangRubin}). That is, Theorem \ref{IMC-refinement} amounts to the assertion that not only the characteristic ideals of the aforementioned modules agree but that in fact their $\Lambda_\chi$-module structures are intimately related. 
\end{rk}

\textit{Proof of Theorem \ref{IMC-refinement}:}
    At the outset we remark that the equivariant Iwasawa Main Conjecture is known to be valid (by the work of Burns and Greither \cite{BurnsGreither}) for the data $(L_\infty | \Q, S \cup S_\infty (\Q), \varnothing, p)$, where $S = \{ v \mid f p \}$, and so Proposition \ref{eimc-proposition} implies that 
    \[
    \NS^b(\ZZ_p(1), L^+_\infty) = \langle e^+ \eta_{f} \rangle_{\bLambda},
    \]
    where for each $m \mid f$ we put $\eta_m = ( 1 - \xi_{m p^{n + 1}})_{n \geq 0} \in U^\infty_S$ and use the idempotent $e^+ = \frac12 (1 + c)$ with complex conjugation $c \in \cG : = \gal{L}{\Q}$. \medskip \\
    Given this, part (a) of Theorem \ref{pairing-theorem} implies that there is an isomorphism
    \begin{align*}
        e^+ \left ( \faktor{U_{S}^\infty}{\langle \eta_{f} \rangle} \right ) \cong
        e^+
        \Ext^1_\bLambda(
        \faktor{\bLambda}{ \Fitt_{\bLambda}^0 
    (H^2_\Iw (\bigO_{L, S}, \Z_p (1))^{\ast \ast}}
    , \ \bLambda).
    \end{align*}
    By assumption the order of $\cG$ is invertible in $\ZZ_p$ and hence $\Lambda_\chi$ is a projective $\bLambda$-module. In particular, the functor $(-)^\chi$ is exact and thus we obtain 
    an isomorphism of $\Lambda_\chi$-modules
    \begin{align*}
        \faktor{U_S^{\infty,\chi}}{\langle e_\chi \eta_{f}\rangle_{\bLambda}} & \cong
        \Ext^1_{\Lambda_\chi}
        \left ( \faktor{\Lambda_\chi}{\Fitt^0_{\Lambda_\chi} (H^2_\Iw (\bigO_{L, S}, \Z_p (1))^\chi)^{\ast \ast}},
        \Lambda_\chi
        \right) \\
        & =
        \Ext^1_{\Lambda_\chi}
        \left ( \faktor{\Lambda_\chi}{\cchar_{\Lambda_\chi} (H^2_\Iw (\bigO_{L, S}, \Z_p (1))^\chi)},
        \Lambda_\chi
        \right)
        .
    \end{align*}

    The explicit description of the pairing given in Theorem \ref{pairing-theorem} then shows that this isomorphism is induced by the map
    \begin{equation} \label{TheIsom}
    e^+ U^{\infty}_S \to e^+ \Hom_{\bLambda} ( \faktor{\bLambda}{\Fitt^0_\bLambda ( H^2_\Iw (\bigO_{L, S}, \Z_p (1))^{**}}, 
    \; \faktor{Q (\bLambda)}{\bLambda} )
    \quad e^+ u \mapsto \{ \lambda \mapsto e^+ (\lambda \cdot \eta_{f}^\ast (u)) \},
    \end{equation}
    where $\eta^\ast_{f} \in Q (\bLambda) \otimes_\bLambda ( U^{\infty}_S)^\ast$ is the dual of $\eta_{f}$. We next note that the Euler system norm relation in this case reads
    \begin{equation} \label{ESNormRelations}
    N_{ \Q ( \xi_{f}) | \Q (\xi_{m})} ( \eta_k)
    = 
    [\Q ( \xi_{f}) : \Q (\xi_{k})] \cdot \Big ( 
    \prod_{\substack{l \mid k \\ l \nmid m}} (1 - \Frob_l^{-1}) \Big ) \cdot \eta_m
    \end{equation}
    for any pair of integers $(k,m)$ satisfiying the divisibility relation $m \mid k \mid f$. Now let $d$ be the conductor of $\chi$, then for $d \mid k \mid f$ we get that
    \[
    e_\chi \eta_k = [\Q ( \xi_{k}) : \Q (\xi_{d})]^{-1} \cdot e_\chi \Big ( 
    \prod_{\substack{l \mid k \\ l \nmid d}} (1 - \Frob_l^{-1}) \Big ) \cdot \eta_d
    \]
    is a $\bLambda$-multiple of $\eta_d$.
    For any $m \nmid d$, on the other hand, we have
    \[
    e_\chi \eta_m = [\Q ( \xi_{f}) : \Q (\xi_{m})]^{-1} \cdot e_\chi N_{\Q ( \xi_{f}) | \Q (\xi_{m})} ( \eta_m) = 0
    \]
    since $\chi$ is non-trivial on $\gal{\Q ( \xi_{f})}{\Q (\xi_m)}$. 
    Let $t = (\gamma - 1)$ for a topological generator $\gamma$ of $\Gamma$, then the above reasoning shows that
    \[
    \Cyc^{\infty, \chi} = 
    \langle e_\chi t^{\delta_\chi} \eta_d \rangle_{\Lambda_\chi} 
   \qquad \text{ for } \delta_\chi = \begin{cases}
   0 & \text{ if } \chi \neq 1, \\
   1 & \text{ if } \chi =1,
   \end{cases}
    \]
    where we have used the fact that, if $d \neq 1$, both $\eta_d, t \eta_1 \in U^\infty$ since $p \nmid d$. 
    Moreover, from (\ref{ESNormRelations}) we deduce that the isomorphism (\ref{TheIsom}) maps $e_\chi \eta_d$ onto the element corresponding to multiplication by $e_\chi \theta_d^{-1}$, where
    \begin{equation} \label{Image}
    \theta_d = 
    \prod_{\substack{l \mid f \\ l \nmid d}} (1 - \Frob_l^{-1}).
       \end{equation}
    Write now $V = S\setminus S_\infty$. Then from the exact sequence
    \begin{cdiagram}
    0 \arrow{r} & 
    A_S^\infty \arrow{r} & H^2_\Iw (\bigO_{L, S}, \Z_p (1)) \arrow{r} & X_V^{\infty} \arrow{r} & 0 
    \end{cdiagram}
    one obtains, using the fact
    that characteristic ideals are multiplicative, an equality
    \begin{align*}
        \cchar_{\Lambda_\chi} ( H^2_\Iw (\bigO_{L, S}, \Z_p (1))^\chi)  
        & = \cchar_{\Lambda_\chi} ( A_S^{\infty, \chi})
        \cdot \cchar_{\Lambda_\chi} ( X^{\infty, \chi}_V ).
    \end{align*}
    An explicit calculation furthermore shows (\textit{c.f.}\@ \cite[Lem. 5.5]{Flach} but note that in our case $p \in S$) that
    \[
    \cchar_{\Lambda_\chi} ( X^{\infty, \chi}_V) = 
    \Big (e_\chi t^{\epsilon_\chi} \prod_{\substack{l \mid f \\ l \nmid d}} (1 - \Frob_l^{-1} ) \Big ) = (e_\chi t^{\epsilon_\chi }\theta_d),
    \]
    where 
    \[
    \epsilon_\chi = \begin{cases} 1 & \text{ if } \chi (p) = 1 \text{ and } \chi \neq 1, \\
    0 & \text{ otherwise}.
    \end{cases}
    \]
      Hence we have the exact sequence
    \begin{cdiagram}[column sep=small]
    0 \arrow{r} & \faktor{\Lambda_\chi}{t^{\epsilon_\chi} \cchar_{\Lambda_\chi} (A_S^{\infty,\chi})} \arrow{r}{\cdot \theta_d} & 
    \faktor{\Lambda_\chi}{\cchar_{\Lambda_\chi} H^2_\Iw (\bigO_{L, S}, \Z_p (1))^\chi} 
    \arrow{r} & 
    \faktor{\Lambda_\chi}{t^{-\epsilon_\chi} \cchar_{\Lambda_\chi} (X_V^\infty)} \arrow{r} & 0
    \end{cdiagram}
    that combines with the statement (\ref{Image}) to give a commutative diagram
    \begin{cdiagram}[column sep=tiny]
    0 \arrow{r} & \faktor{\langle e_\chi \eta_d \rangle}{\langle e_\chi \eta_{f} \rangle} 
    \arrow{d}{\simeq} \arrow{r} &
    \faktor{U^{\infty, \chi}_S}{\langle e_\chi \eta_{f} \rangle} 
    \arrow{r} \arrow{d}{\simeq} & 
    \faktor{U^{\infty, \chi}_S}{\langle e_\chi \eta_d \rangle}
\arrow[dashed]{d} \arrow{r} & 
0 \\
0 \arrow{r} & \alpha \left(
\displaystyle\faktor{\Lambda_\chi}{t^{-\epsilon_\chi} \cchar_{\Lambda_\chi} (X_V^\infty)} \right ) 
\arrow{r} &
\alpha \left(
\displaystyle
 \faktor{\Lambda_\chi}{\cchar_{\Lambda_\chi} (H^2_\Iw (\bigO_{L, S}, \Z_p (1))^\chi)} 
 \right )
 \arrow{r}{\cdot \theta_d}
 &
 \alpha \left(
 \displaystyle
 \faktor{\Lambda_\chi}{t^{\epsilon_\chi} \cchar_{\Lambda_\chi} (A_S^{\infty, \chi})}
 \right)
 \arrow{r} & 
 0
    \end{cdiagram}

    An application of the snake lemma then implies that the right hand map is an isomorphism.
    Now, observe that the module 
     $\faktor{U^{\infty, \chi}}{\langle e_\chi t^{\delta_\chi} \eta_d \rangle} =  \faktor{U^{\infty, \chi}}{\Cyc^{\infty, \chi}}$ is $\Lambda_\chi$-cyclic since $U^{\infty, \chi}$ is $\Lambda_\chi$-free of rank one. The aforementioned quotient is furthermore $\Z_p$-torsion free as it injects into
     $\faktor{U_S^{\infty, \chi}}{\langle e_\chi \eta_d \rangle} \cong \alpha (  \faktor{\Lambda_\chi}{t^{\epsilon_\chi} \cchar_{\Lambda_\chi} (A_S^{\infty, \chi})})$. It follows that 
     \begin{align*}
     \Ann_{\Lambda_\chi} \left ( \faktor{U^{\infty, \chi}}{\Cyc^{\infty, \chi}} \right ) & = 
     \Fitt^0_{\Lambda_\chi} \left ( \faktor{U^{\infty, \chi}}{\Cyc^{\infty, \chi}} \right ) = \cchar_{\Lambda_\chi} \left ( \faktor{U^{\infty, \chi}}{\Cyc^{\infty, \chi}} \right ) \\
     & = 
     \cchar_{\Lambda_\chi} ( A^{\infty, \chi} ),
     \end{align*}
    where we have used the classical Iwasawa Main Conjecture \cite[Thm. 5.1]{LangRubin} to establish the final equality. By using an explicit description analogous to (\ref{TheIsom}), we see that the image of $\faktor{U^{\infty, \chi}}{\Cyc^{\infty, \chi}}$
    under the isomorphism 
    \begin{equation} \label{AnotherIsom}
    \faktor{U^{\infty, \chi}_S}{\langle e_\chi \eta_d \rangle}
    \cong 
    \alpha \left(
 \displaystyle
 \faktor{\Lambda_\chi}{t^{\epsilon_\chi} \cchar_{\Lambda_\chi} (A_S^{\infty, \chi})}
 \right)
    \end{equation}
    coincides with the kernel of multiplication by a generator of $ \cchar_{\Lambda_\chi} ( A^{\infty, \chi} )$ on $\alpha \left(
 \faktor{\Lambda_\chi}{\cchar_{\Lambda_\chi} (A_S^{\infty, \chi})}
 \right)$. An argument entirely similar to the one utilised above show that this kernel is exactly $\alpha (
 \faktor{\Lambda_\chi}{\cchar_{\Lambda_\chi} (A^{\infty, \chi})})$.
Hence the isomorphism (\ref{AnotherIsom}) restricts to give the isomorphism claimed in the statement of Theorem \ref{IMC-refinement}. 
    \qed

\subsubsection{Greenberg's conjecture}

In this subsection we will rely on the results of Appendix \ref{AppendixIwasawa}. We resume the notation and assumptions of \S\,\ref{set-up-section}.

\begin{prop} \label{GreenbergCriterion}
Assume the $\mu$-invariant of $A_{S, \Sigma}^\infty$ vanishes. The following are equivalent:
\begin{liste}
\item The module $A_{S, \Sigma}^\infty$ is finite,
\item there is an isomorphism
\[
\faktor{\NS^{r_T} ( \Z_p (1), L_\infty)}{\NS^b (\Z_p (1), L_\infty)} \cong \Ext^1_\bLambda \Big ( \faktor{\bLambda}{\Fitt^0_\bLambda ( X_{S \setminus S_\infty}^\infty)},\bLambda \Big).
\]
\end{liste}
\end{prop}

\begin{rk} \label{greenberg-rk}
\begin{liste}
\item If $L$ is a totally real field, we may take $\Sigma = \emptyset$ since $p$ is odd. Then $A_{S, \Sigma}^\infty$ agrees with the inverse limit of $S$-class groups $A_S^\infty$. In this situation, Greenberg has conjectured \cite{Greenberg} that both the $\mu$ and $\lambda$-invariant of the $\Lambda$-module $A^\infty$ vanish, i.e.\@ that $A^\infty$ is finite, and this clearly implies that $A_S^\infty$ is finite. 
In certain situations on can show that the converse of this implication holds. This is the case, for example, if either $L$ validates Leopoldt's conjecture or has exactly one prime that ramifies in $L_\infty$ (see \cite[\S\,4]{Iwasawa}). 
\item Theoretical evidence for Greenberg's conjecture is still very sparse. The only general class of fields known to satisfy the conjecture are fields $L$ with a unique prime above $p$ and such that $A_{S_p} (L) = 1$ (in this case Greenberg's conjecture follows from Nakayama's Lemma). However, there are many explicit examples giving evidence for the conjecture, starting with Greenberg's original article \cite[\S\,8]{Greenberg}. For example, Kraft and Schoof \cite{KraftSchoof} have numerically verified the conjecture for $p = 3$ and all real quadratic fields $\Q ( \sqrt{f})$ such that $f \not \equiv 1 \mod 3$ and $f < 10000$. 
\item The case $r_T = 1$ of Proposition \ref{GreenbergCriterion} is classical and well-known. To the best of the knowledge of the authors, a result of this shape first appeared in \cite[Lem.\@ 1]{Gold}.
\end{liste}
\end{rk}

\textit{Proof of Proposition \ref{GreenbergCriterion}:}
We have an exact sequence
\begin{cdiagram}
   0 \arrow{r} & A_{S, \Sigma}^\infty \arrow{r} & H^2_{\Sigma, \Iw} (\bigO_{L, S}, \Z_p (1)) \arrow{r} & X^\infty_{S \setminus S_\infty} \arrow{r} & 0
\end{cdiagram}
and it follows that at every regular height one prime $\p$ of $\bLambda$ there is an equality
\begin{equation} \label{factorisation}
\Fitt^0_\bLambda ( H^2_{\Sigma, \Iw} (\bigO_{L, S}, \Z_p (1)))_\p = \Fitt^0_\bLambda (A_{S, \Sigma}^\infty)_\p \cdot \Fitt^0_\bLambda ( X^\infty_{S \setminus S_\infty} )_\p
\end{equation}
since $\bLambda_\p$ is a discrete valuation ring in this case. If $\p$ is a singular prime, in turn, then $H^2_{\Sigma, \Iw} (\bigO_{L, S}, \Z_p (1))_\p = \Fitt^0_\bLambda (A_{S, \Sigma}^\infty)_\p$ by Lemma \ref{IwasawaInvariants}\,(a) because $X^\infty_{S \setminus S_\infty}$ has vanishing $\mu$-invariant. \medskip \\
Let us now assume that $A_{S, \Sigma}^\infty$ is finite. The previous discussion combines with Lemma \ref{IwasawaInvariants}\,(b) to imply that $\Fitt^0_\bLambda ( X^\infty_{S \setminus S_\infty})_\p = \Fitt^0_\bLambda (H^2_{\Sigma, \Iw} (\bigO_{L, S}, \Z_p (1)))_\p$ for all height-one primes $\p$ of $\bLambda$. This implies that the surjection
\[
\faktor{\bLambda}{\Fitt^0_\bLambda (H^2_{\Sigma, \Iw} (\bigO_{L, S}, \Z_p (1)))} \to \faktor{\bLambda}{\Fitt^0_\bLambda ( X^\infty_{S \setminus S_\infty})}
\]
is a pseudo-isomorphism and hence has finite kernel. 
It follows that the induced map
\begin{align*}
\Ext^1_\bLambda \Big ( \faktor{\bLambda}{\Fitt^0_\bLambda ( X^\infty_{S \setminus S_\infty})},\bLambda \Big ) & \longrightarrow \Ext^1_\bLambda \Big ( \faktor{\bLambda}{\Fitt^0_\bLambda (H^2_{\Sigma, \Iw} (\bigO_{L, S}, \Z_p (1)))},\bLambda \Big ) \\
& \longrightarrow \Ext^1_\bLambda \Big ( \faktor{\bLambda}{\Fitt^0_\bLambda (H^2_{\Sigma, \Iw} (\bigO_{L, S}, \Z_p (1)))^{\ast \ast}},\bLambda \Big )
\end{align*}
is an isomorphism, where we have used that the second arrow is also induced by a pseudo-isomorphism (see the discussion in the proof of Theorem \ref{pairing-theorem}\,(c)). Taking $\Ext^1_\bLambda ( - , \bLambda)$ across the isomorphism in Theorem \ref{pairing-theorem}\,(a) then gives (cf.\@ (\ref{double-ext})) 
\[
\faktor{\NS^{r_T}}{\NS^b} \cong \Ext^1_\bLambda \Big ( \faktor{\bLambda}{\Fitt^0_\bLambda ( X^\infty_{S \setminus S_\infty})},\bLambda \Big ). 
\]
Conversely, assume to be given such an isomorphism. Taking $\Ext^1_\bLambda ( - , \bLambda)$ and combining with Theorem \ref{pairing-theorem}, we obtain a pseudo-isomorphism
\begin{align} \label{pseudo-isomorphism}
\faktor{\bLambda}{\Fitt^0_\bLambda (H^2_{\Sigma, \Iw} (\bigO_{L, S}, \Z_p (1)))} \approx \faktor{\bLambda}{\Fitt^0_\bLambda ( X^\infty_{S \setminus S_\infty})}. 
\end{align}
Now let $\p$ be a regular height one prime, then the above pseudo-isomorphism gives 
\[
\Fitt^0_\bLambda (H^2_{\Sigma, \Iw} (\bigO_{L, S}, \Z_p (1)))_\p \cong 
\Fitt^0_\bLambda ( X^\infty_{S \setminus S_\infty})_\p.
\]
Since $\bLambda_\p$ is a discrete valuation ring, this implies
\begin{align*}
    (\p \bLambda_\p)^{\text{length}_{\bLambda_\p} (X^\infty_{S \setminus S_\infty})_\p + \text{length}_{\bLambda_\p} (A_{S, \Sigma}^\infty)_\p } & = \Fitt^0_\bLambda (A^\infty_{S, \Sigma})_\p \cdot \Fitt^0_\bLambda ( X^\infty_{S \setminus S_\infty} )_\p \\
    & = \Fitt^0_\bLambda ( X^\infty_{S \setminus S_\infty} )_\p \\
    & = (\p \bLambda_\p)^{\text{length}_{\bLambda_\p} (X^\infty_{S \setminus S_\infty})_\p}.
\end{align*}
We deduce that $\text{length}_{\bLambda} (A^\infty_{S, \Sigma})_\p = 0$, hence $(A^\infty_{S, \Sigma})_\p = 0$. Thus, the finiteness of $A_{S, \Sigma}^\infty$ follows now from 
the assumed vanishing of its $\mu$-invariant and
Lemma \ref{IwasawaInvariants}. 
\qed

\subsubsection{Leading term conjectures}

We continue using the notations an assumptions of \S\,\ref{set-up-section}. In this section we describe a connection between Proposition \ref{GreenbergCriterion} and conjectures concerning the leading terms of equivariant $L$-functions that appear in the literature. The central player in these conjectures is the \textit{$S$-truncated and $\Sigma$-modified Dirichlet $L$-function} that is defined as
\[
L_{L_n | K, S, \Sigma} (\chi, s) = \prod_{v \in \Sigma} (1 - \chi ( \Frob_v) \text{N}v^{1 - s} ) \cdot \prod_{v \not \in S} (1 - \chi (\Frob_v) \text{N}v^{-s} )^{-1}
\]
for any complex values $s$ satisfying $\text{Re} (s) > 1$, and any character $\chi \in \widehat{\cG_n}$. It is well known that $L_{L_n | K, S, \Sigma} (\chi, s)$ can be continued to a meromorphic function that is defined on the whole complex plane and holomorphic at $s = 0$. For any $r \geq 0$, we denote the $r^{th}$-th coefficient in the Taylor expansion of $L_{L_n | K, S, \Sigma}  (\chi, s)$ at $s = 0$ by 
\[
L^{(r)}_{L_n | K, S, \Sigma} (\chi, 0) = \lim_{s \to 0} s^{-r} L_{L_n | K, S, \Sigma}  (\chi, s). 
\]
and define the \textit{Stickelberger element} to be
\[
\theta_{L_n | K, S, \Sigma}^{(r)} (0) = \sum_{\chi \in \widehat{\cG_n}} e_{\overline{\chi}} L_{L_n | K, S, \Sigma}^{(r)} (0).
\]
Fix an isomorphism $\C \cong \C_p$ and recall that the Dirichlet regulator defines a $\C_p [\cG_n]$-linear isomorphism
\[
\lambda_{L_n, S, \Sigma} \: \C_p \otimes_\Z \bigO_{L_n, S, \Sigma}^\times 
\stackrel{\simeq}{\longrightarrow}
\C_p \otimes_\Z X_{L_n, S}, 
\quad
a \mapsto - \sum_{w \in S_{L_n}} \log (| a|_w ) \cdot w.
\]
We remind the reader that the integer $r_T \geq 0$ for $T = \Z_p (1)$ is given by $r_T = | S_\infty (K) |$ under the running hypotheses. For every $v \in S$, we also fix a place $w \in S_{L_n}$ such that $w | v$. 
\begin{definition}
Pick $w_0 \in S \setminus S_\infty$.
The $r_T$-th \textit{Rubin-Stark element} $\varepsilon_{L_n | K, S, \Sigma}$ is the preimage of $\theta^{(r_T)}_{L_n | K, S, \Sigma} (0) \cdot \bigwedge_{w \in S_\infty} (w - w_0)$
under the isomorphism
\[
\C_p \otimes_{\Z_p} \exprod^{r_T}_{\Z_p [\cG_n]} U_{L_n, S, \Sigma} \stackrel{\simeq}{\longrightarrow} \C_p \otimes_{\Z_p} \exprod^{r_T}_{\Z_p [\cG_n]} X_{L_n, S}
\]
induced by the Dirichlet regulator $\lambda_{L, S, \Sigma}$.
\end{definition}

We shall investigate the following conjectures.

\begin{conjecture} \label{conjectures}
    \begin{liste}
    \item We have an inclusion
    \[
    \varepsilon_{L_n | K, S, \Sigma} \in \bidual^{r_T}_{\Z_p [\cG_n]} U_{L_n, S, \Sigma}.
    \]
    \item There is a basis $\mathfrak{z}_{L_n} \in \Det_{\Z_p [\cG_n]} ( C^\bullet_n)$ such that $\Theta_{L_n} (\mathfrak{z}_{L_n}) = \varepsilon_{L_n | K, S, \Sigma}$. 
    \item There is a basis $\mathfrak{z}_{\infty} \in \Det_\bLambda (C^\bullet_\infty)$ such that $\Theta (\mathfrak{z}_\infty) = (\varepsilon_{L_n | K, S, \Sigma})_n$. 
    \end{liste}
\end{conjecture}

\begin{rk} \label{conjectures-remark}
The above are special cases of conjectures appearing in the literature. Conjecture \ref{conjectures}\,(a) is the relevant case of the Rubin-Stark conjecture \cite[Conj.\@ B']{Rubin96} in this setting. Conjecture \ref{conjectures}\,(b) is a consequence of the equivariant Tamagawa number conjecture as stated, for example, in \cite[Conj.\@ 2.3]{BKS2} after taking \cite[Prop.\@ 2.5]{BKS2} into consideration. It is easy to see that if $L | K$ has a unique place above $p$ and $S = S_\infty (K) \cup S_p$, then Conjecture \ref{conjectures}\,(b) in fact coincides with \cite[Conj.\@ 2.3]{BKS2}. Finally, Conjecture \ref{conjectures}\,(c) is the higher rank equivariant Iwasawa main conjecture appearing in \cite[Conj.\@ 3.1]{BKS2}. 
\end{rk}

The following is an analogue of the classical index formula for cyclotomic units (see \cite[Thm.\@ 8.2]{Washington}).

\begin{lem} \label{IndexFormula}
Assume that all infinite places split in $L | K$
and that there is a unique place $\p$ above $p$ in $L$. Put $S = S_\infty (K) \cup S_p$
and assume that $A_{S, \Sigma} (L)$ is $\cG$-cohomologically trivial. Then
\[
\big (\bidual^{r_T}_{\Z_p [\cG]} U_{L, S, \Sigma} \; : \; \Z_p [\cG] \cdot \varepsilon_{L | K, S, \Sigma} \big ) = h_{L, S, \Sigma}, 
\]
where $h_{L, S, \Sigma} = | A_{S, \Sigma} ( L) |$ is the $p$-part of the $(S, \Sigma)$-class number of $L$. 
\end{lem}

\begin{proofbox}
Since $\cG$ acts trivially on $\p$, the module $X_{L, S}$ is $\Z_p [\cG]$-free, generated by the $\Z_p [\cG]$-linearly independent set $\{ (w - \p) \}$, where $w$ ranges over our set of fixed places in $S_L$. It follows that $H^2 (C^\bullet_0)$ is $\cG$-cohomologically trivial and the isomorphism
\[
\widehat{H}^i (\cG, \ U_{L, S, \Sigma})\cong \widehat{H}^{i + 2} ( \cG, \ H^2 (C^\bullet_0))
\]
for all $i \in \Z$ induced by the complex $C^\bullet_0$ shows that $U_{L,S, \Sigma }$ is $\cG$-cohomologically trivial, hence $\Z_p [\cG]$-projective. The Dirichlet regulator map $\lambda_{L, S, \Sigma}$ induces a rational isomorphism $\Q_p U_{L, S, \Sigma} \cong \Q_p X_{L, S}$ (see \cite[\S\,I.4.3]{Tate}), which in this setting (see \cite[Thm.\@ 5.6.10\,(ii)]{NSW}) implies that there is an isomorphism $U_{L, S, \Sigma} \cong X_{L, S}$ and so $U_{L,S,\Sigma}$ is $\ZZ_p[\cG]$-free. Let $\{u_1, \dots, u_{r_T}\}$ be any $\ZZ_p[\cG]$-basis of this module. Then 
\[
\bidual^{r_T}_{\Z_p [\cG]} U_{L, S, \Sigma} = \exprod^{r_T}_{\Z_p [\cG]} U_{L, S, \Sigma} = \Z_p [\cG] \cdot (u_1 \wedge \dots \wedge u_{r_T}).
\]
It now suffices to calculate the index  
\[
\Z_p [\cG] \cdot \lambda_{L, S, \Sigma} ( \varepsilon_{L | K, S, \Sigma}) = \Z_p [\cG] \cdot \theta^{(r_T)}_{L, S, \Sigma} (0) 
\quad\text{ inside } \quad
\Z_p [\cG] \cdot \lambda_{L, S, \Sigma} ( u_1 \wedge \dots \wedge u_{r_T}) 
\]
as sublattices of $\C_p \exprod_{\Z_p [\cG]}^r X_{L, S} \cong \C_p [\cG]$. If we can find a $\C_p$-linear isomorphism $f \: \C_p [\cG] \to \C_p [\cG]$ that maps the first of these two aforementioned lattices bijectively onto the latter, then by \cite[Lem.\@ 1.1\,(b)]{Sinnott} this index is given by the $\C_p$-determinant of $(\det f)^{-1}$. \medskip \\
Define $f$ to be the $\C_p [\cG]$-linear extension of $1 \mapsto \theta^{(r_T)}_{L, S, \Sigma} (0)^{-1} \cdot  \lambda_{L, S, \Sigma} ( u_1 \wedge \dots \wedge u_{r_T})$, then this map has the desired properties. Calculating the determinant of multiplication by $\theta^{(r_T)}_{L, S, \Sigma} (0)^{-1}$ with respect to the basis $\{ e_\chi \}_{\chi\in \widehat{\cG}}$, we find that it equals
\[ \prod_{\chi \in \widehat{G}} L^{(r_T)}_{L | K, S, \Sigma} (\overline{\chi}, 0)^{-1} = \zeta_{L, S, T}^{(r_T)} (0)^{-1} = (h_{L, S, \Sigma}\cdot R_{L, S, \Sigma})^{-1}
\]
by the analytic class number formula, where $R_{L,S,\Sigma}$ is the $(S,\Sigma)$-regulator. Finally, using \cite[Ch.\@ III, \S\,9.4, Prop.\@ 6]{Bourbaki}, we conclude that multiplication by $\lambda_{L, S, \Sigma} ( u_1 \wedge \dots \wedge u_{r_T})$ has determinant $R_{L, S, \Sigma}$. This completes the proof of the Lemma.
\end{proofbox}

\begin{thm} \label{etnc-thm}
    Let $L$ be a totally real field. Assume $|S_p (L)| = 1$ and that $p \nmid |\cG|$. If 
    \begin{enumerate}[label=(\roman*)]
        \item the Rubin-Stark conjecture \ref{conjectures}\,(a) holds for the data $(L_n | K, S, \varnothing)$ for all $n \in \N_0$,
        \item Greenberg's conjecture holds for $L$,
    \end{enumerate}
    then the equivariant Iwasawa Main Conjecture \cite[Conj.\@ 3.1]{BKS2} for $L_\infty | K$ and the equivariant Tamagawa Number Conjecture \cite[Conj.\@ 2.3]{BKS2} for $L | K$ both hold true for $S = S_\infty \cup S_p$.
\end{thm}

\begin{proofbox}
By \cite[Prop.\@ 6.1]{Rubin96} we have $\varepsilon_{L | K, S} \in \UN_0^r$, so Lemma \ref{IndexFormula} in particular implies that 
the index of $\UN^{r_T}_0$ inside $\bidual^{r_T}_{\Z_p [\cG]} U_{L, S}$ is finite. By Proposition \ref{Dodgy-Proposition}\,(b) and \cite[Prop.\@ 7]{CornacchiaGreither}, this index is given by
\begin{align*}
    \left( \bidual^{r_T}_{\Z_p [\cG]} U_{L, S} \ : \ \UN^{r_T}_0 \right) &
    = ( \Z_p [\cG] : \Fitt^0_{\Z_p [\cG]} ((A_{S}^\infty)^{\Gamma, \vee} )) = | (A_{S}^\infty)^{\Gamma, \vee} | = | (A_{S}^\infty)^{\Gamma} |,
\end{align*}
where we have used that $A_{S}^\infty$ is finite by assumption (ii). This finiteness also implies that the Herbrand quotient of $A_{S}^\infty$ is trivial, hence
\[
| (A_{S}^\infty)^\Gamma | = | (A_{S}^\infty)_\Gamma | = | A_{S} (L) | = h_{L, S},
\]
where the second equality follows from applying \cite[Prop.\@ 13.22]{Washington}. Now, Lemma \ref{IndexFormula} implies that
\[
\UN^{r_T}_0 = \Z_p [\cG] \cdot \varepsilon_{L | K, S}. 
\]
It follows from Nakayama's Lemma that the sequence $(\varepsilon_{L_n | K, S})_{n \geq 0}$ is a $\bLambda$-basis of $\NS^{r_T} ( L_\infty, \Z_p (1))$. 
Since we are assuming (ii), Proposition \ref{GreenbergCriterion} gives the equality $\NS^{r_T} = \NS^b$, hence the equivariant Iwasawa Main Conjecture \ref{conjectures}\,(c) holds true. Moreover, we know from Lemma \ref{FiniteLem} that $\im \Theta_L = \UN^b_0$, so we also find that $\UN^{r_T}_0 = \im \Theta_L$. Thus, Conjecture \ref{conjectures}\,(b) is valid. Since we have already observed in Remark \ref{conjectures-remark} that in this setting Conjecture \ref{conjectures}\,(b) coincides with \cite[Conj.\@ 2.3]{BKS2}, this concludes the proof. 
\end{proofbox}

\subsection{$\mu$-invariant conjectures and the Tate module of elliptic curves}\label{mu-section}

In this section we use the result of Theorem \ref{pairing-theorem} to give a reformulation of the various $\mu$-vanishing conjectures. To do this, we assume that $L_\infty | L$ is the cyclotomic $\ZZ_p$-extension and consider the following Hypothesis:
\begin{hypothesis}\label{sigma-hypothesis}
    $\Sigma$ is a finite (possibly empty) set of places of $K$, disjoint from $S$, such that for every $v \in \Sigma$ the module of invariants $H^0(K_v, T)$ vanishes.
\end{hypothesis}
We caution the reader that we are not yet assuming that $\Sigma$ is chosen in such a way that Hypothesis \ref{main-hypothesis}\,(3) and (4) are satisfied.\medskip \\
It is then natural to formulate the following conjecture.

\begin{conjecture}\label{mu-vanishing-conjecture}
    Assume that $\Sigma$ satisfies Hypothesis \ref{sigma-hypothesis}. Then $H^2_{\Sigma, \Iw}(\cO_{L,S},T)$ is a torsion $\Lambda$-module and, furthermore, has vanishing $\mu$-invariant as a $\Lambda$-module.
\end{conjecture}

\begin{lemma}\label{mu-vanishing-independent-lemma}
    Conjecture \ref{mu-vanishing-conjecture} is independent of the choice of $\Sigma$.
\end{lemma}

\begin{proof}
    By the definition of $\Sigma$-modified cohomology one has, for each $n \in \NN_0$, an exact sequence
    \begin{cdiagram}
        \displaystyle \bigoplus_{w \in \Sigma_{L_n}} H^1 (\kappa_w, T) \arrow{r} & H^2_\Sigma(\cO_{L_n, S}, T) \arrow{r} &  H^2(\cO_{L_n, S}, T) \arrow{r} &
        0
    \end{cdiagram}
    Fix a place $v \in \Sigma$ and recall that the complex $\bigoplus_{w \mid v} \text{R}\Gamma_\et (\kappa_w,T)$ is represented in $D(\cR[\cG_n])$ by the complex $\bigoplus_{w \mid v}\; [T \xrightarrow{1-\Frob_w^{-1}} T]$.
    In particular, Hypothesis \ref{sigma-hypothesis} implies that $\bigoplus_{w \mid v} H^1 (\kappa_w, T)$ is finite and its minimal number of $\cR$-generators is bounded above by the $\cR$-rank of $\oplus_{w\mid v} T$.\\
    Since $L_\infty | L$ is the cyclotomic $\ZZ_p$-extension, there are only finitely many primes of $L_\infty$ lying above $v$. We may thus pass to the limit to deduce that $\bigoplus_{w \in \{v\}_{L_\infty}}\varprojlim_n H^1 (\kappa_w,T)$ is finitely generated as a $\Z_p$-module, hence has vanishing $\mu$-invariant.
    
    This fact now combines with the exact sequence above to imply that $H^2_{\Sigma, \Iw}(\cO_{L, S}, T)$ is $\Lambda$-torsion and, moreover, has vanishing $\mu$-invariant if and only if the same is true of $H^2_{\Iw}(\cO_{L, S}, T)$.
\end{proof}

\begin{examples}\label{mu-vanishing-examples}Lemma \ref{mu-vanishing-conjecture} shows that, under the current hypotheses, the question of whether $H^2_{\Sigma, \Iw}(\cO_{L,S}, T)$ is a torsion $\Lambda$-module is equivalent to the weak Leopoldt conjecture for the pair $(T^*(1),L_\infty)$. As for the $\mu$-invariant component of Conjecture \ref{mu-vanishing-conjecture} we make the following observations:
    \begin{liste}
        \item{When $K$ is totally real and $T = \ZZ_p(1)$, then Hypothesis \ref{sigma-hypothesis} is satisfied for any choice of $\Sigma$. The exact sequence
        \begin{cdiagram}
            0 \arrow{r} & A_{S, \Sigma}^\infty \arrow{r} &
            H^2_{\Sigma, \Iw} (\cO_{L,S},T) \arrow{r} &
            X_{S}^\infty \arrow{r} & 0
        \end{cdiagram}
        combines with Lemma \ref{mu-vanishing-independent-lemma} to imply that Conjecture \ref{mu-vanishing-conjecture} is equivalent to Iwasawa's famous conjecture on the vanishing of the $\mu$-invariant of $A^\infty$. In particular, the theorem of Ferrero-Washington implies that Conjecture \ref{mu-vanishing-conjecture} is valid whenever $L | \QQ$ is an abelian extension.
        }
        \item{Let $T = T_p(E)$ be the $p$-adic Tate module of an elliptic curve over $K$, and take $\Sigma$ to be a set of primes of $K$ disjoint from $S$. Since $E$ has good reduction at every prime in $\Sigma$ it follows that $\Sigma$ satisfies Hypothesis \ref{sigma-hypothesis}.\\
        As such, Conjecture \ref{mu-vanishing-conjecture} is equivalent to the Coates-Sujatha conjecture \cite[Conj. A]{coates-sujatha} after taking into consideration Lemma 3.2 of \textit{loc.\@ cit.}
        In particular, if $K = \QQ$ and $L | \QQ$ is any finite abelian extension such that $E(L)[p^\infty] \neq 0$ then Conjecture \ref{mu-vanishing-conjecture} is valid by \cite[Cor.\@ 3.6]{coates-sujatha}.}
        
        \item{More generally, Lim has conjectured in \cite[Conj.\@ A]{lim} that the $\mu$-invariant of $H^2_\Iw(\cO_{L,S}, T)$ vanishes and Conjecture \ref{mu-vanishing-conjecture} is equivalent to his conjecture after taking into consideration Lemma 3.4 of \textit{loc.\@ cit.}}
    \end{liste}
\end{examples}

We can now formulate the main result of this section.

\begin{proposition}\label{mu-vanishing-result}
    Suppose that $\Sigma$ is chosen to satisfy Hypothesis \ref{sigma-hypothesis} and so that the triple $(T,L_\infty, \Sigma)$ satisfies Hypothesis \ref{main-hypothesis}. Then Conjecture \ref{mu-vanishing-conjecture} is valid if and only if $\NS^{r_T}/\NS^b$ is finitely generated as an $\cR$-module.
\end{proposition}

\begin{proof}
    At the outset we first note that if $M$ is a finitely generated torsion $\bLambda$-module, then its $\mu$-invariant vanishes if and only if $M_\p = 0$ for every singular prime $\p$ of $\bLambda$ by Lemma \ref{IwasawaInvariants}.
    To prove the Proposition we now fix such a prime $\p$ of $\bLambda$. By localising the isomorphism of Theorem \ref{pairing-theorem} at $\p$ we obtain an isomorphism
    \begin{align*}
        \Ext_{\bLambda}^1\left(\faktor{\NS^{r_T}}{\NS^b}, \bLambda\right)_\fp \cong \faktor{\bLambda_\p}{\Fitt_{\bLambda_\p}(H^2_{\Sigma, \Iw}(\cO_{L,S}, T)_\p)}.
    \end{align*}
    On the other hand, one knows by \cite[Prop.\@ 5.5.13]{NSW} that $\faktor{\NS^{r_T}}{\NS^b}$ and $\Ext^1_\bLambda\left(\faktor{\NS^{r_T}}{\NS^b}, \bLambda\right)$ have the same $\mu$-invariant. 
    As such, one deduces from the above discussion that the $\mu$-invariant of $H^2_{\Sigma, \Iw}(\cO_{L,S}, T)$ vanishes if and only if the same is true of that of $\faktor{\NS^{r_T}}{\NS^b}$.
\end{proof}

Using this Proposition we can now appeal to the known validity of the Coates-Sujatha conjecture for particular elliptic curves over $\QQ$ to say something about the structure of the quotient module $\NS^{r_T}/\NS^b$.

\begin{corollary}\label{mu-vanishing-example}
    Let $K= \QQ$ and $E/\QQ$ be an elliptic curve. If $(T,L_\infty, \Sigma)$ satisfies Hypothesis \ref{main-hypothesis} and $E(L)[p^\infty] \neq 0$, then $\NS^1/\NS^b$ is a free $\ZZ_p$-module of finite rank.
\end{corollary}

\begin{proof}
    By combining Proposition \ref{mu-vanishing-result} with the observation of Example \ref{mu-vanishing-examples}\,(b) one deduces that $\NS^1/\NS^b$ is finitely generated as a $\ZZ_p$-module. Moreover, $\NS^1/\NS^b$ has projective dimension one as a $\Lambda$-module and so has no non-zero finite $\Lambda$-submodules. Since the maximal finite $\Lambda$-submodule of $\NS^1/\NS^b$ necessarily coincides with its maximal finite $\ZZ_p$-submodule it then follows that $\NS^1/\NS^b$ is $\ZZ_p$-free as claimed.
\end{proof}

\begin{remark}
    Let $K = L = \QQ$ and let $E/\QQ$ be an elliptic curve of algebraic rank 0 with finite Tate-Shaferevich group and such that $E(\QQ) \neq 0$. Suppose that $\Sigma$ is chosen so that $(T, \QQ_\infty, \Sigma)$ satisfies Hypothesis \ref{main-hypothesis}. If $p > 7$ does not divide any of the Tamagawa numbers at primes of bad reduction for $E$ nor the order of the Tate-Shaferevich group of $E$, then Wuthrich has shown in \cite[Prop.\@ 9.1]{wuthrich} that the fine Selmer group of $E$ over $\QQ_\infty$ is trivial.
    If we assume, in addition, that for every $v \in S\setminus S_\infty(\QQ)$ the group $E(\QQ_v)$ has no points of order $p$, then a straightforward argument using the Weil pairing implies that the dual of the fine Selmer group of $E$ over $\QQ_\infty$ coincides with $H^2_\Iw(\cO_{\Q,S}, T)$.
    Given this, the proof of Proposition \ref{mu-vanishing-result} implies that, under all the above assumptions, $\NS^1(T,\QQ_\infty) = \NS^b(T,\QQ_\infty)$.
\end{remark}

\markboth{Appendix}{Finite places splitting in $L_\infty$}
\section*{Appendix}
\addcontentsline{toc}{section}{Appendix}

\renewcommand{\thesubsection}{\Alph{subsection}}
\renewcommand{\thethm}{(\Alph{subsection}.\arabic{thm})}
\setcounter{subsection}{0}
\setcounter{thm}{0}

\subsection{Finite places splitting in $L_\infty$}

Key to the approach of the main body of the present article is the assumption that no finite place of $K$ splits completely in $L_\infty$. While this is not an incredibly strict assumption (it is always satisfied for the cyclotomic $\ZZ_p$-extension of $L$, for example), it is natural to ask whether one can weaken this hypothesis to include many other interesting situations that arise in arithmetic. 
For example, if $K = L$ is an imaginary quadratic field and $K_\infty$ is not the anti-cyclotomic $\ZZ_p$-extension of $K$, then only finitely many finite places of $K$ can split in $K_\infty$ (see, for example, \cite{emsalem}). 
In particular, to obtain a module of higher rank universal norms in this situation which incorporates all arithmetic data of interest one must modify the notion of basic rank given above to take into account these new non-archimedean places.\medskip \\
In this appendix we briefly outline how one can can do this for general $p$-adic representations. Since this exposition will contain essentially no additional ideas we prefer to prove as little as we feel is necessary and refer to existing arguments to justify the claims.\medskip \\
At the outset we adopt all notations in \S\,\ref{set-up-section} with the exception that $L_\infty | L$ is now a $\ZZ_p$-extension in which finitely many places (infinite or finite) of $K$ are allowed to split.

\tocless\subsubsection{$\Sigma$-modified Selmer complexes}

In this subsection we introduce a $\Sigma$-modified Selmer complex which shall prove useful throughout this appendix.\medskip \\
Denote by $S_f$ the subset of $S$ comprising the finite places. Let $V$ be a finite set of finite places of $K$, disjoint from $\Sigma$. For any abelian extension $F$ of $K$ we define the \textit{$\Sigma$-modified $V$-supported} complex $\text{R}\Gamma_{\Sigma,V}(\cO_{F,S}, T)$ to be the mapping fibre in $D(\cR[\gal{F}{K}])$ of the morphism
\begin{cdiagram}
    \text{R}\Gamma(\cO_{F,S},T) \oplus \bigoplus_{w \in (S_f\setminus V)_F} \text{R}\Gamma(F_w, T) \arrow{r}{\phi} & \bigoplus_{w \in S_{f,F}} \text{R}\Gamma(F_w, T) \oplus \bigoplus_{w \in \Sigma_{F}} \text{R}\Gamma(\kappa_w, T),
\end{cdiagram}
where $\phi$ is given by $(\oplus_{w \in S_{f,F}} -\res_w, \oplus_{w \in \Sigma_{F}} -\res_w)$ on the former summand and by $(\iota, 0)$ on the latter summand where $\iota$ is the natural inclusion map. \medskip \\
Given this construction, we then have the following Lemma which results from a straightforward application of the octahedral axiom.

\begin{lemma}
    Suppose that $V \subseteq S$. Then there is an exact triangle
    \begin{equation}\label{selmer-complex-exact-triangle}
    \begin{tikzcd}
        \text{R}\Gamma_{\Sigma,V}(\cO_{F,S}, T) \arrow{r} &  \text{R}\Gamma_{\Sigma}(\cO_{F,S}, T) \arrow{r} &  \bigoplus_{w \in V_F} \text{R}\Gamma(F_w, T)
        \end{tikzcd}
    \end{equation}
    in $D(\cR[\cG_F])$.
\end{lemma}

We now use this Lemma to study the complex $C_n^\bullet$ defined in \S\,\ref{set-up-section}.

\begin{proposition} \label{SplitPrimesProp}
    Fix $n \in \NN_0$. Denote by $V_n$ the subset of $S$ consisting of those finite places that split completely in $L_n$ and write
    \begin{align*}
        Y_{n,K}(T) := \qquad \mathclap{\bigoplus_{v \in S_\infty(K) \cup V_n}} \quad H^0(K_v, T^*(1)).
    \end{align*}
    Then there is an exact sequence
    \begin{cdiagram}
        H^2_{\Sigma,V_n}(\cO_{L_n, S}, T) \arrow{r} & H^1(C^\bullet_n) \arrow{r} & \bigoplus_{v \in S_\infty(K) \cup V_n}  Y_{n,K}(T)^* \otimes_{\cR} \cR[\cG_n] \arrow{r} &  0
    \end{cdiagram}
    in which the first map is canonical and the second depends on a choice of a set of representatives of the orbits of $\Gal(L_n | K)$ on $V_n$.
\end{proposition}

\begin{proof}
    From Proposition \ref{FiniteComplex} there is a decomposition
    \begin{align*}
        H^1(C_n^\bullet) \cong H^2_{\Sigma}(\cO_{L_n, S}, T) \oplus \bigoplus_{v \in S_\infty(K)} H^0(K_v, T^\ast(1))^\ast.
    \end{align*}
    On the other hand, Tate local duality combines with the fact that $p$ is odd and the exact triangle (\ref{selmer-complex-exact-triangle}) to imply the existence of an exact sequence
    \begin{cdiagram}
        H^2_{\Sigma, V_n}(\cO_{L_n, S}, T) \arrow{r} &  H^2_{\Sigma}(\cO_{L_n, S}, T) \arrow{r} & \bigoplus_{v \in V_n} H^0(K_v, T^\ast(1))^\ast \otimes_\cR \cR[\cG_n] \arrow{r} &  0
    \end{cdiagram}
    in which the middle arrow depends on a choice of a set of representatives of the orbits of $\Gal(L_n/K)$ on $V_n$. In particular, if we write $K_n$ for the kernel of the left hand map then, since every prime in $V_n$ splits completely in $L_n$, we obtain a decomposition
    \begin{align*}
        H^1(C_n^\bullet) \cong K_n \oplus (Y_{n,K}(T)^* \otimes_{\cR} \cR[\cG_n])
    \end{align*}
    which completes the proof of the claim. 
\end{proof}

\tocless\subsubsection{Universal norms}

We write $S_\spc(L_\infty)$ for the set of places of $K$ that split completely in $L_\infty$ and remark that, by our assumptions and the standard properties of $\ZZ_p$-extensions, $S_\spc(L_\infty)$ contains all the archimedean places of $K$. We then fix a finite set of places $S$ containing
\begin{align*}
    S_\spc(L_\infty) \cup S_p(K) \cup S_\ram(T) \cup S_\ram(L/K)
\end{align*}
and we suppose that the tuple $(T,L_\infty, \Sigma)$ satisfies the following mild hypotheses: \medskip

\begin{hypotheses}\label{main-hypothesis-appendix}\text{}
    \begin{enumerate}[label=(\arabic*)]
        \item{For every $n \in \NN_0$ one has that the module of invariants $H^0_\Sigma(L_n, T)$ vanishes.}
        \item{The $\cR$-free module $Y_K(T) = \bigoplus_{v \in S_\spc(L_\infty)} H^0(K_v,T^*(1))$ has non-zero rank $r_T$.}
        \item{$H^1_\Sigma(\cO_{L_n, S}, T)$ is $\cR$-torsion-free for every $n \in \N_0$}.
    \end{enumerate}
\end{hypotheses}

For every subextension $F$ of $L_\infty | K$ we write
\begin{align*}
    H^2_{\Sigma,\spc}(\cO_{F, S}, T) := H^2(R\Gamma_{\Sigma, S_\spc(L_\infty)}(\cO_{F,S}, T)).
\end{align*}

Now fix $n \in \NN_0$. By considering the sum of primitive idempotents in $\cQ[\cG_n]$ that annihilate the kernel $K_n := \ker(H^2_{\Sigma,\spc}(\cO_{L_n,S}, T) \to H^1(C_n^\bullet))$ of the map of Proposition \ref{SplitPrimesProp} one can construct, as in \S\ref{finite-level-section}, a projection map
\begin{align*}
    \Theta_{L_n} \: \Det_{\cR[\cG_n]}(C_{L_n}^\bullet) \to \bidual_{\cR[\cG_n]}^{r_T} H^1_\Sigma(\cO_{L_n,S},T).
\end{align*}

In particular, if there exists $n \in \NN$ for which the projection map $\Theta_{L_n}$ is non-zero, then the argument of \cite[Prop.\@ 4.30\,(i)]{EulerSystemsSagaIII} can be used to show that $\varprojlim_n K_n$ is a torsion $\bLambda$-module.\medskip \\
Given this, it is straightforward to check that the arguments used to prove Theorem \ref{UN-structure-theorem} and Theorem \ref{pairing-theorem} give analogous results under the running hypotheses of this appendix.

\subsection{Equivariant Iwasawa algebras}
\label{AppendixIwasawa}

This appendix serves the purpose of collecting useful facts on rings of the form $\bLambda = \cR \llbracket \Gamma \rrbracket [G]$, where 
\begin{itemize}
\item $\cR$ is the ring of integers in a finite extension $\mathcal{Q}$ of $\mathbb{Q}_p$,
\item $\Gamma \cong \Z_p$
\item $G$ is a finite abelian group.
\end{itemize}

We shall also write $\Gamma^n$ for the unique subgroup of $\Gamma$ of order $p^n$ and $\Gamma_n$ for the quotient $\faktor{\Gamma}{\Gamma^n}$. Moreover, we set $\Lambda = \cR \llbracket \Gamma \rrbracket$. 

\tocless\subsubsection{Basic properties}

Let $P$ be the $p$-Sylow subgroup of $G$ and write $H = \faktor{G}{P}$. Then we have a character-part decomposition
\begin{equation}
    \bLambda \cong \bigoplus_{\chi \in \widehat{H} / \sim} \Lambda_\chi \quad \text{ with } \Lambda_\chi = \cR [\im \chi] \llbracket \Gamma \rrbracket [P].\label{equivariant-decomposition}
\end{equation}
Here the relation $\sim$ is defined as $\chi \sim \chi'$ if there is a $\sigma \in G_\mathcal{Q}$ such that $\chi = \sigma \circ \chi'$. \medskip

\begin{lem} \label{LambdaGorenstein}
The ring $\bLambda$ is Gorenstein.
\end{lem}

\begin{proofbox}
Let $\gamma$ be a topological generator of $\Gamma$. For every $\chi \in \widehat{H} / \sim$, the quotient $\faktor{\Lambda_\chi}{(\gamma - 1)} \cong \bigO [\im \chi] [P]$ is a finite group ring and is therefore Gorenstein by \cite[\@10.29]{CurtisReiner}. Since $\gamma - 1$ is a non-zero divisor in $\Lambda_\chi$, this implies that $\Lambda_\chi$ is Gorenstein as well. \\
If $\p \subseteq \bLambda$ is any prime ideal, then the localisation $\bLambda_\p$ is given by the sum $\bigoplus_{\chi \in \widehat{H} / \sim} (\Lambda_\chi)_{\p_\chi}$, where $\p_\chi$ denotes the projection of $\p$ onto $\Lambda_\chi$. Each summand $(\Lambda_\chi)_{\p_\chi}$ is already a local ring, so we must have $\bLambda_\p = (\Lambda_\chi)_{\p_\chi}$ for a certain character $\chi$. Since $\Lambda_\chi$ is a Gorenstein ring, this shows that $\bLambda$ is Gorenstein as well. 
\end{proofbox}

We also note that for any commutative ring $R$ and $R$-module $M$, we have a natural isomorphism of $R[G]$-modules
\begin{align}
\Hom_R (M, R) \xrightarrow{\sim} \Hom_{R[G]} (M, R[G])^\#, \quad f \mapsto \Big \{ m \mapsto \sum_{\sigma \in G} \sigma^{-1} f(\sigma \cdot m) \Big \}\label{hom-isomorphism}
\end{align}
In particular, when $R = \Lambda$ we obtain an isomorphism
\begin{equation} \label{Exts}
\Ext^i_\Lambda (M, \Lambda) \cong \Ext^i_\bLambda (M, \bLambda)^\# \quad \text{ for all } i \geq 0. 
\end{equation}

\tocless\subsubsection{Height-one primes} \label{HeightOnePrimesSection}

The aim of this section is to describe the relation between the classical Iwasawa $\lambda$ and $\mu$-invariants of a module $M$ and the localisation of $M$ at height one primes of $\bLambda$. In order to do this, it is useful to make the following distinction among the latter. 
\begin{definition}
A height-one prime $\p$ of $\bLambda$ is called \emph{singular} if $p \in \p$ and \emph{regular} otherwise. 
\end{definition}

If $\p$ is regular, then $\bLambda_\p$ is identified with the localisation of $\bLambda [\tfrac{1}{p}]$ at $\p \bLambda [\tfrac{1}{p}]$. From the decomposition
\[
\bLambda [\tfrac{1}{p}] = \bigoplus_{\chi \in \widehat{G} / \sim} \Lambda_\chi [\tfrac{1}{p}] \qquad
\text{ with } \Lambda_\chi [\tfrac{1}{p}] = \bigO [\im \chi] \llbracket \Gamma \rrbracket [\tfrac{1}{p}]
\]
we deduce via an argument similar to the one used in the proof of Lemma \ref{LambdaGorenstein} that the localisation $(\bLambda [\tfrac{1}{p}])_{\p \bLambda [\frac{1}{p}]}$ agrees with $(\Lambda_\chi [\tfrac{1}{p}])_{\p\Lambda_\chi [\frac{1}{p}] }$ for a certain character $\chi$. 
Note that $\Z_p [\im \chi] \llbracket \Gamma \rrbracket$ is a conventional Iwasawa algebra. In particular, its localisation $\Z_p [\im \chi] \llbracket \Gamma \rrbracket [\tfrac{1}{p}]$ is also a regular local domain. It follows that $\bLambda_\p$ is a discrete valuation ring. 
\\
If $\p$ is such that $\bLambda_\p = (\Lambda_\chi [\tfrac{1}{p}])_{\p\Lambda_\chi [\frac{1}{p}] }$ for the character $\chi \in \widehat{G} / \sim$, then we say the character $\chi$ is \textit{associated} to $\p$. Similarly, if $\p$ is a singular prime and $\bLambda_\p = (\Lambda_\chi)_{\p \Lambda_\chi}$ for some character $\chi \in \widehat{H} /\sim$, then we call $\chi$ \text{associated} to $\p$ as well. It should always be clear from the context what is meant.

\begin{lem} \label{IwasawaInvariants}
Let $M$ be a finitely generated $\bLambda$-torsion module.
\begin{liste}
\item The $\mu$-invariant of $M$ (as a $\Lambda$-module) vanishes if and only if $M_\p = 0$ for every singular prime $\p$ of $\bLambda$.
\item The module $M$ is finite (equivalently, its $\mu$ and $\lambda$-invariant as a $\Lambda$-module both vanish) if and only if $M_\p = 0$ for every height-one prime $\p$ of $\bLambda$. 
\end{liste}
\end{lem}

\begin{proofbox}
The `only if' part of (a) follows from \cite[Lem.\@ 5.6]{Flach} (see also \cite[Lem.\@ 6.3]{BurnsGreither}).\\
As for the converse implication, we write $M_\chi := M \otimes_\bLambda \Lambda_\chi$ for each $\chi \in \widehat{H}/\sim$. Then by tensoring the decomposition (\ref{equivariant-decomposition}) with $M$ we have a natural isomorphism $M \cong \bigoplus_{\chi \in \widehat{H}/\sim} M_\chi$. 
Now, for any character $\chi \in \widehat{H}/\sim$ there exists a singular prime $\p$ of $\bLambda$ such that $\chi$ is associated to $\p$. The hypothesis now implies that the module $M_{\fp} = M_{\chi, \fp}$ vanishes, hence $M_\chi$ is a finitely generated $\cO(\im(\chi))$-module by the result of \textit{loc.\@ cit}. Since $\cO(\im(\chi))$ is finitely generated as an $\cO$ module it then follows that $M$ is finitely generated as an $\cO$-module whence it has vanishing $\mu$-invariant as a $\Lambda$-module.\medskip \\
For part (b) we first note that if the $\lambda$-invariant of $M$ vanishes then $M$ is necessarily $\ZZ_p$-torsion (\textit{c.f.}\@ \cite[Rem.\@ 3 following Def.\@ 5.3.9]{NSW}). It then follows that $M_\p = 0$ for any regular height-one prime $\p$ of $\bLambda$. In combination with (a) this gives one direction of assertion (b). \\
Conversely, assume that $M_\p = 0$ for any height-one prime $\p$ of $\bLambda$. Appealing to \cite[Lemma B.11]{Sakamoto20} one sees that
\[
0 = \Ext^1_\bLambda (M, \bLambda ) \cong \Ext^1_\Lambda (M, \Lambda ). 
\]
This combines with \cite[Prop.\@ 5.5.8\,(iv)]{NSW} to imply that $M$ is finite as desired.
\end{proofbox}

\renewcommand{\emph}[1]{\textit{#1}}

\pagestyle{special}
\vspace{-1.5em}
\addcontentsline{toc}{section}{References}
\tiny
\printbibliography

\small


\textsc{King's College London,
Department of Mathematics,
London WC2R 2LS,
United Kingdom} \\
\textit{Email address:} \href{mailto:dominik.bullach@kcl.ac.uk}{dominik.bullach@kcl.ac.uk}\\

\textsc{King's College London,
Department of Mathematics,
London WC2R 2LS,
United Kingdom}\\
\textit{Email address:} \href{mailto:alexandre.daoud@kcl.ac.uk}{alexandre.daoud@kcl.ac.uk}

\end{document}